\documentclass[fullpage, oneside]{amsart}

\usepackage{amsmath, amssymb}
\usepackage{latexsym}
\usepackage[english]{babel}\usepackage{hyperref}

\usepackage{amssymb}
\usepackage{amsmath}
\usepackage{amsthm}

\usepackage{a4wide}
\usepackage{amsthm,amssymb,amsmath,a4wide,graphicx,tikz}

\usepackage{subfigure}

\usepackage{caption}

\usepackage{hhline}
\usetikzlibrary{patterns}
\usepackage{float}
\usepackage{dsfont}
\usepackage{mathtools}
\DeclarePairedDelimiter\floor{\lfloor}{\rfloor}

\usepackage{titlesec}
\titlelabel{\thetitle.\quad}

\usepackage{siunitx}

\usepackage[a4paper]{geometry}
\DeclareMathOperator{\Li}{Li}

\makeatletter
\renewcommand*\env@matrix[1][*\c@MaxMatrixCols c]{%
  \hskip -\arraycolsep
  \let\@ifnextchar\new@ifnextchar
  \array{#1}}
\makeatother

\definecolor{apricot}{rgb}{0.98, 0.81, 0.69}
\definecolor{aquamarine}{rgb}{0.5, 1.0, 0.83}
\definecolor{babyblueeyes}{rgb}{0.63, 0.79, 0.95}
\definecolor{bananamania}{rgb}{0.98, 0.91, 0.71}
\definecolor{bittersweet}{rgb}{1.0, 0.44, 0.37}
\definecolor{bluebell}{rgb}{0.64, 0.64, 0.82}
\definecolor{blush}{rgb}{0.87, 0.36, 0.51}

\definecolor{cerulean}{rgb}{0.0, 0.48, 0.65}
\definecolor{darkcyan}{rgb}{0.0, 0.55, 0.55}

\definecolor{airforceblue}{rgb}{0.36, 0.54, 0.66}
\definecolor{antiquefuchsia}{rgb}{0.57, 0.36, 0.51}
\definecolor{asparagus}{rgb}{0.53, 0.66, 0.42}
\definecolor{ballblue}{rgb}{0.13, 0.67, 0.8}
\definecolor{blueviolet}{rgb}{0.54, 0.17, 0.89}
\definecolor{brightgreen}{rgb}{0.4, 1.0, 0.0}
\definecolor{brightcerulean}{rgb}{0.11, 0.67, 0.84}
\definecolor{brightpink}{rgb}{1.0, 0.0, 0.5}
\definecolor{brightturquoise}{rgb}{0.03, 0.91, 0.87}
\definecolor{brightube}{rgb}{0.82, 0.62, 0.91}
\definecolor{brilliantlavender}{rgb}{0.98, 0.86, 1.0}
\definecolor{bluebell}{rgb}{0.64, 0.64, 0.82}
\definecolor{bluegray}{rgb}{0.4, 0.6, 0.8}
\definecolor{bluegreen}{rgb}{0.0, 0.87, 0.87}
\definecolor{brandeisblue}{rgb}{0.0, 0.44, 1.0}
\definecolor{capri}{rgb}{0.0, 0.75, 1.0}
\definecolor{fandango}{rgb}{0.71, 0.2, 0.54}

\definecolor{bubblegum}{rgb}{0.99, 0.76, 0.8}
\definecolor{brilliantrose}{rgb}{1.0, 0.33, 0.64}
\definecolor{brightmaroon}{rgb}{0.76, 0.13, 0.28}

\newtheorem{thm}{Theorem}[section]
\newtheorem{lem}[thm]{Lemma}

\newtheorem{prop}[thm]{Proposition}
\newtheorem{nrem}[thm]{Remark}

\begin{document}
\title{Matching for a family of infinite measure continued fraction transformations}
\author[Charlene Kalle]{Charlene Kalle$^\dagger$}
\author[Niels Langeveld] {Niels Langeveld$^\dagger$}
\author[Marta Maggioni]{Marta Maggioni$^\dagger$}
\author[Sara Munday]{Sara Munday$^\ddagger$}

\address[$\dagger$]{Department of Mathematics, Leiden University,
Niels Bohrweg 1, 2333CA Leiden, The Netherlands}
\email[Charlene Kalle]{kallecccj@math.leidenuniv.nl}
\email[Niels Langeveld]{n.d.s.langeveld@math.leidenuniv.nl}
\email[Marta Maggioni]{m.maggioni@math.leidenuniv.nl}
\address[$\ddagger$]{John Cabot University, Via della Lungara, 233, 00165 Roma, Italy}
\email[Sara Munday]{smunday@johncabot.edu}


\begin{abstract}
As a natural counterpart to Nakada's $\alpha$-continued fraction maps, we study a one-parameter family of continued fraction transformations with an indifferent fixed point. We prove that matching holds for Lebesgue almost every parameter in this family and that the exceptional set has Hausdorff dimension 1. Due to this matching property, we can construct a planar version of the natural extension for a large part of the parameter space. We use this to obtain an explicit expression for the density of the unique infinite $\sigma$-finite absolutely continuous invariant measure and to compute the Krengel entropy, return sequence and wandering rate of the corresponding maps.
\end{abstract}
\maketitle

\section{Introduction}
Over the past decades the dynamical phenomenon of matching, or synchronisation, has surfaced increasingly often in the study of the dynamics of interval maps, mostly due to the fact that systems with matching exhibit properties that resemble those of Markov maps. A map $T$ is said to have the matching property if for any discontinuity point $c$ of the map $T$ or its derivative $T'$ the orbits of the left and right limits of $c$ eventually meet. That is, there exist non-negative integers $M$ and $N$, called \textit{matching exponents}, such that
\begin{equation}\label{q:matchingdef}
T^M(c^-)=T^N(c^+),
\end{equation}
where
\[c^-=\lim_{x \uparrow c} T(x) \quad \text{and } \quad c^+= \lim_{x \downarrow c} T(x).\]
\vskip .2cm
General results on the implications of matching are scarce. There are many results however on the consequences of matching for specific families of interval maps. In \cite{KS12, BORG13, BCK,BCMP,CM18,DK} matching was considered for various families of piecewise linear maps in relation to expressions for the invariant densities, entropy and multiple tilings. Another type of transformation for which matching has proven to be convenient is for continued fraction maps, most notably for Nakada's $\alpha$-continued fraction maps. This family was introduced in \cite{Nak81} by defining for each $\alpha \in \big[\frac12,1\big]$ the map $S_\alpha:[\alpha-1, \alpha] \to [\alpha-1, \alpha]$ by $S_\alpha(0)=0$ and for $x \neq 0$,
\begin{equation}\label{q:nakada}
S_\alpha(x) = \Big| \frac1x \Big| - \Big\lfloor  \Big| \frac1x \Big| +1-\alpha  \Big\rfloor.
\end{equation}
In \cite{Nak81} Nakada constructed a planar natural extension of $S_\alpha$ and proved the existence of a unique absolutely continuous invariant probability measure. In \cite{LM08} the family was extended to include the parameters $\alpha \in \big[0, \frac12\big)$. On this part of the parameter space the planar natural extension strongly depends on the matching property, and it is much more complicated. This also affects the behaviour of the metric entropy as a function of $\alpha$, which is described in detail in \cite{LM08, NN08, CMPT10, KSS12, CT13,T14}. In \cite{DKS09,KU10, CIT18} matching was successfully considered for other families of continued fraction transformations.

\vskip .2cm
The matching behaviour of these different families has some striking similarities. The parameter space usually breaks down into maximal intervals on which the exponents $M$ and $N$ from \eqref{q:matchingdef} are constant, called \textit{matching intervals}. These matching intervals usually cover most of the space, leaving a Lebesgue null set. The set where matching fails, called the \textit{exceptional set}, is often of positive Haussdorff dimension, see \cite{CT12, KSS12, BCIT13, BCK, DK} for example.

\vskip .2cm
So far, matching has been considered only for dynamical systems with a finite absolutely continuous invariant measure. In this article, we introduce and study the matching behaviour and its consequences for a one-parameter family of continued fraction transformations on the interval that have a unique absolutely continuous, $\sigma$-finite invariant measure that is infinite. This family of \textit{flipped $\alpha$-continued fraction transformations} we introduce arises naturally as a counterpart to Nakada's $\alpha$-continued fraction maps. Due to matching we obtain a nice planar natural extension on a large part of the parameter space, which allows us to explicitly compute dynamical features of the maps, such as the invariant density, Krengel entropy and wandering rate.

\vskip .2cm
The family of maps $\{ T_\alpha \}_{\alpha \in (0,1)}$ we consider is defined as follows. For each $\alpha \in (0,1)$ let
\begin{equation}\label{q:Dalpha}
 D_\alpha = \bigcup_{n \ge 1} \Big[ \frac1{n+\alpha}, \frac 1n \Big] \subseteq [0,1],
 \end{equation}
and $I_\alpha := [\min\{\alpha, 1-\alpha\},1]$, and define the map $T_\alpha: I_\alpha  \to I_\alpha$ by
\begin{equation}\nonumber
T_\alpha(x)= \left\{
  \begin{array}{l l}
    G(x) =  \displaystyle \frac{1}{x}-\left\lfloor\frac{1}{x}\right\rfloor,		& \text{if }x \in D_\alpha^{\mathsf{c}},\\[0.5cm]
    1-G(x) =  \displaystyle -\frac{1}{x} + \left(1+\left\lfloor\frac{1}{x}\right\rfloor\right),	&\text{if } x\in D_\alpha,
\end{array}  \right.
\end{equation}
where $G(x) = \frac1x \pmod 1$ is the {\em Gauss map} and $D_\alpha^{\mathsf{c}}$ denotes the complement of $D_{\alpha}$ in $[0,1]$. Note that for $\alpha=0$ one recovers the Gauss map $G$, and $\alpha=1$ gives $1-G$, which is a shifted version of the {\em R\'enyi map} or {\em backwards continued fraction map}. Since these transformations have already been studied extensively, we omit them from our analysis. Figures~\ref{f:off}(c) and~\ref{f:off}(f) show the graphs of the maps $T_{\alpha}$ for a parameter $\alpha < \frac12$ and a parameter $\alpha> \frac12$, respectively. We could define $T_\alpha$ on the whole interval $[0,1]$, but since the dynamics of $T_{\alpha}$ is attracted to the interval $I_\alpha$ we just take that as the domain. Since $I_\alpha$ is bounded away from 0, any map $T_\alpha$ has only a finite number of branches. Note also that each map $T_\alpha$ has an indifferent fixed point at $1$.

\vskip .2cm
We call these transformations flipped $\alpha$-continued fraction maps, due to their relation to the family of maps described in \cite{MMY97}. The authors defined for each $\alpha \in [0,1]$ the {\em folded $\alpha$-continued fraction map} $\hat S_\alpha: [0, \max\{\alpha, 1-\alpha \}] \to [0, \max \{ \alpha, 1-\alpha \}]$ by $\hat S_\alpha(0)=0$ and for $x \neq 0$,
\[ \hat S_\alpha (x) = \Big| \frac1x - \Big\lfloor \frac1x +1-\alpha \Big\rfloor \Big| = \left\{
  \begin{array}{l l}
    G(x) =  \displaystyle \frac{1}{x}-\left\lfloor\frac{1}{x}\right\rfloor,		& \text{if }x \in D_\alpha,\\[0.5cm]
    1-G(x) =  \displaystyle -\frac{1}{x} + \left(1+\left\lfloor\frac{1}{x}\right\rfloor\right),	&\text{if } x \in D_\alpha^{\mathsf{c}}.
\end{array}  \right. \]
The dynamical properties of the folded $\alpha$-continued fraction maps are essentially equal to those of Nakada's $\alpha$-continued fraction maps. The name represents the idea that these maps `fold' the interval $[\alpha-1,\alpha]$ onto $[0, \max\{ \alpha, 1-\alpha\}]$. As shown in Figure~\ref{f:off} the families $\{ T_\alpha\}$ and $\{ \hat S_\alpha\}$ are obtained by flipping the Gauss map on complementary parts of the unit interval and, as such, both families are particular instances of what are called {\em $D$-continued fraction maps} in \cite{DHKM12}. Furthermore, for $\alpha=\frac12$ the transformation $T_\alpha$ seems to be closely related, but not isomorphic, to the object of study of \cite{DK00}.

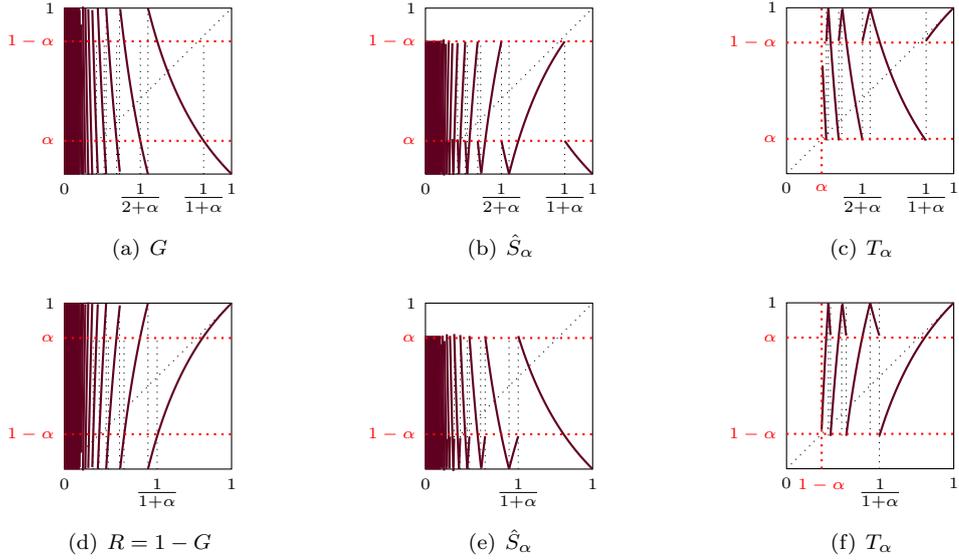
\begin{figure}[h]
\centering
\subfigure[$G$]{\begin{tikzpicture}[scale=2.2]
\draw[white] (-.6,0)--(1.5,0);
\draw(0,0)node[below]{\tiny $0$}--(1,0)node[below]{\tiny $1$}--(1,1)--(0,1)node[left]{\tiny $1$}--(0,0);
 \draw[dotted] (.5,0)--(.5,1);
  \draw[dotted] (.33,0)--(.33,1);
  \draw[dotted] (.25,0)--(.25,1);
 \draw[dotted] (0,0)--(1,1);
 \draw[dotted] (.833,0)node[below]{\small $\frac1{1+\alpha}$}--(.833,.8);
 \draw[dotted] (.4545,0)node[below]{\small $\frac1{2+\alpha}$}--(.4545,.8);
 \draw[dotted] (.3125,0)--(.3125,.8)(.238,0)--(.238,.8)(.192,0)--(.192,.8);

 \draw[thick, purple!50!black, smooth, samples =20, domain=.5:1] plot(\x,{1 / \x -1});
\draw[thick, purple!50!black, smooth, samples =20, domain=.334:.5] plot(\x,{1 /\x -2});
\draw[thick, purple!50!black, smooth, samples =20, domain=.25:.332] plot(\x,{1 /\x -3});
\draw[thick, purple!50!black, smooth, samples =20, domain=.2:.25] plot(\x,{1 /\x -4});
\draw[thick, purple!50!black, smooth, samples =20, domain=.167:.2] plot(\x,{1 /\x -5});
\draw[thick, purple!50!black, smooth, samples =20, domain=.143:.167] plot(\x,{1 /\x -6});
\draw[thick, purple!50!black, smooth, samples =20, domain=.125:.143] plot(\x,{1 /\x -7});
\draw[thick, purple!50!black, smooth, samples =20, domain=.111:.125] plot(\x,{1 /\x -8});
\draw[thick, purple!50!black, smooth, samples =20, domain=.101:.1112] plot(\x,{1 /\x -9});
\draw[thick, purple!50!black, smooth, samples =20, domain=.091:.1] plot(\x,{1 /\x -10});
\filldraw[purple!50!black] (0,0) rectangle (.09,1);
\draw[thick, red, dotted](0,.2)node[left]{\tiny $\alpha$}--(1,.2);
\draw[thick, red, dotted](0,.8)node[left]{\tiny $1-\alpha$}--(1,.8);
\end{tikzpicture}}
\subfigure[$\hat S_\alpha$]{\begin{tikzpicture}[scale=2.2]
\draw[white] (-.6,0)--(1.5,0);
\draw(0,0)node[below]{\tiny $0$}--(1,0)node[below]{\tiny $1$}--(1,1)--(0,1)node[left]{\tiny $1$}--(0,0);
 \draw[dotted] (.5,0)--(.5,.8);
  \draw[dotted] (.33,0)--(.33,.8);
  \draw[dotted] (.25,0)--(.25,.8);
 \draw[dotted] (0,0)--(1,1);
  \draw[dotted] (.833,0)node[below]{\small $\frac1{1+\alpha}$}--(.833,.8);
 \draw[dotted] (.4545,0)node[below]{\small $\frac1{2+\alpha}$}--(.4545,.8);
 \draw[dotted] (.3125,0)--(.3125,.8)(.238,0)--(.238,.8)(.192,0)--(.192,.8);

\draw[thick, purple!50!black, smooth, samples =20, domain=.833:1] plot(\x,{1 / \x -1});
\draw[thick, purple!50!black, smooth, samples =20, domain=.5:.833] plot(\x,{2-1 / \x });

\draw[thick, purple!50!black, smooth, samples =20, domain=.4545:.5] plot(\x,{1 /\x -2});
\draw[thick, purple!50!black, smooth, samples =20, domain=.334:.4545] plot(\x,{3-1 /\x});

\draw[thick, purple!50!black, smooth, samples =20, domain=.3125:.332] plot(\x,{1 /\x -3});
\draw[thick, purple!50!black, smooth, samples =20, domain=.25:.3125] plot(\x,{4-1 /\x });

\draw[thick, purple!50!black, smooth, samples =20, domain=.238:.25] plot(\x,{1 /\x -4});
\draw[thick, purple!50!black, smooth, samples =20, domain=.2:.238] plot(\x,{5-1 /\x });

\draw[thick, purple!50!black, smooth, samples =20, domain=.192:.2] plot(\x,{1 /\x -5});
\draw[thick, purple!50!black, smooth, samples =20, domain=.167:.192] plot(\x,{6-1 /\x });

\draw[thick, purple!50!black, smooth, samples =20, domain=.161:.167] plot(\x,{1 /\x -6});
\draw[thick, purple!50!black, smooth, samples =20, domain=.143:.161] plot(\x,{7-1 /\x });

\draw[thick, purple!50!black, smooth, samples =20, domain=.139:.143] plot(\x,{1 /\x -7});
\draw[thick, purple!50!black, smooth, samples =20, domain=.125:.139] plot(\x,{8-1 /\x });

\draw[thick, purple!50!black, smooth, samples =20, domain=.122:.125] plot(\x,{1 /\x -8});
\draw[thick, purple!50!black, smooth, samples =20, domain=.111:.122] plot(\x,{9-1 /\x });

\draw[thick, purple!50!black, smooth, samples =20, domain=.109:.1112] plot(\x,{1 /\x -9});
\draw[thick, purple!50!black, smooth, samples =20, domain=.101:.109] plot(\x,{10-1 /\x });

\draw[thick, purple!50!black, smooth, samples =20, domain=.098:.1] plot(\x,{1 /\x -10});
\draw[thick, purple!50!black, smooth, samples =20, domain=.091:.098] plot(\x,{11-1 /\x });

\filldraw[purple!50!black] (0,0) rectangle (.09,.8);
\draw[thick, red, dotted](0,.2)node[left]{\tiny $\alpha$}--(1,.2);
\draw[thick, red, dotted](0,.8)node[left]{\tiny $1-\alpha$}--(1,.8);
\end{tikzpicture}}
\subfigure[$T_\alpha$]{\begin{tikzpicture}[scale=2.2]
\draw[white] (-.6,0)--(1.5,0);
\draw(0,0)node[below]{\tiny $0$}--(1,0)node[below]{\tiny $1$}--(1,1)--(0,1)node[left]{\tiny $1$}--(0,0);
 \draw[dotted] (.5,.2)--(.5,1);
  \draw[dotted] (.33,.2)--(.33,1);
  \draw[dotted] (.25,.2)--(.25,1);
 \draw[dotted] (0,0)--(1,1);
  \draw[dotted] (.835,.2)--(.835,1);
  \node at (.835,-.15){\small $\frac1{1+\alpha}$};
 \draw[dotted] (.4545,.2)--(.4545,1);
 \node at (.4545,-.15){\small $\frac1{2+\alpha}$};
 \draw[dotted] (.3125,.2)--(.3125,1)(.238,.2)--(.238,1);

\draw[thick, purple!50!black, smooth, samples =20, domain=.833:1] plot(\x,{2-1 / \x });
\draw[thick, purple!50!black, smooth, samples =20, domain=.5:.833] plot(\x,{1 / \x -1});

\draw[thick, purple!50!black, smooth, samples =20, domain=.4545:.5] plot(\x,{3-1 /\x });
\draw[thick, purple!50!black, smooth, samples =20, domain=.334:.4545] plot(\x,{1 /\x-2});

\draw[thick, purple!50!black, smooth, samples =20, domain=.3125:.332] plot(\x,{4-1 /\x });
\draw[thick, purple!50!black, smooth, samples =20, domain=.25:.3125] plot(\x,{1 /\x -3});

\draw[thick, purple!50!black, smooth, samples =20, domain=.238:.25] plot(\x,{5-1 /\x });
\draw[thick, purple!50!black, smooth, samples =20, domain=.215:.238] plot(\x,{1 /\x -4});

\draw[thick, red, dotted](0,.21)node[left]{\tiny $\alpha$}--(1,.21);
\draw[thick, red, dotted](0,.79)node[left]{\tiny $1-\alpha$}--(1,.79);
\draw[thick, red, dotted] (.21,0)node[below]{\tiny $\alpha$} -- (.21,1);
\end{tikzpicture}}
\subfigure[$R=1-G$]{\begin{tikzpicture}[scale=2.2]
\draw[white] (-.6,0)--(1.5,0);
\draw(0,0)node[below]{\tiny $0$}--(1,0)node[below]{\tiny $1$}--(1,1)--(0,1)node[left]{\tiny $1$}--(0,0);
 \draw[dotted] (.5,0)--(.5,1);
  \draw[dotted] (.33,0)--(.33,1);
  \draw[dotted] (.25,0)--(.25,1);
 \draw[dotted] (0,0)--(1,1);
 \draw[dotted] (.555,0)node[below]{\small $\frac1{1+\alpha}$}--(.555,.8);
 \draw[dotted] (.357,0)--(.357,.8);
 \draw[dotted] (.263,0)--(.263,.8)(.208,0)--(.208,.8)(.172,0)--(.172,.8);

 \draw[thick, purple!50!black, smooth, samples =20, domain=.5:1] plot(\x,{2-1 / \x });
\draw[thick, purple!50!black, smooth, samples =20, domain=.334:.5] plot(\x,{3-1 /\x });
\draw[thick, purple!50!black, smooth, samples =20, domain=.25:.332] plot(\x,{4-1 /\x });
\draw[thick, purple!50!black, smooth, samples =20, domain=.2:.25] plot(\x,{5-1 /\x });
\draw[thick, purple!50!black, smooth, samples =20, domain=.167:.2] plot(\x,{6-1 /\x });
\draw[thick, purple!50!black, smooth, samples =20, domain=.143:.167] plot(\x,{7-1 /\x });
\draw[thick, purple!50!black, smooth, samples =20, domain=.125:.143] plot(\x,{8-1 /\x });
\draw[thick, purple!50!black, smooth, samples =20, domain=.111:.125] plot(\x,{9-1 /\x });
\draw[thick, purple!50!black, smooth, samples =20, domain=.101:.1112] plot(\x,{10-1 /\x });
\draw[thick, purple!50!black, smooth, samples =20, domain=.091:.1] plot(\x,{11-1 /\x });

\filldraw[purple!50!black] (0,0) rectangle (.09,1);
\draw[thick, red, dotted](0,.21)node[left]{\tiny $1-\alpha$}--(1,.21);
\draw[thick, red, dotted](0,.79)node[left]{\tiny $\alpha$}--(1,.79);
\end{tikzpicture}}
\subfigure[$\hat S_\alpha$]{\begin{tikzpicture}[scale=2.2]
\draw[white] (-.6,0)--(1.5,0);
\draw(0,0)node[below]{\tiny $0$}--(1,0)node[below]{\tiny $1$}--(1,1)--(0,1)node[left]{\tiny $1$}--(0,0);
 \draw[dotted] (.5,0)--(.5,.8);
  \draw[dotted] (.33,0)--(.33,.8);
  \draw[dotted] (.25,0)--(.25,.8);
 \draw[dotted] (0,0)--(1,1);
  \draw[dotted] (.555,0)node[below]{\small $\frac1{1+\alpha}$}--(.555,.8);
 \draw[dotted] (.357,0)--(.357,.8);
 \draw[dotted] (.263,0)--(.263,.8)(.208,0)--(.208,.8)(.172,0)--(.172,.8);

\draw[thick, purple!50!black, smooth, samples =20, domain=.555:1] plot(\x,{1 / \x -1});
\draw[thick, purple!50!black, smooth, samples =20, domain=.5:.555] plot(\x,{2-1 / \x });

\draw[thick, purple!50!black, smooth, samples =20, domain=.357:.5] plot(\x,{1 /\x -2});
\draw[thick, purple!50!black, smooth, samples =20, domain=.334:.357] plot(\x,{3-1 /\x});

\draw[thick, purple!50!black, smooth, samples =20, domain=.263:.332] plot(\x,{1 /\x -3});
\draw[thick, purple!50!black, smooth, samples =20, domain=.25:.263] plot(\x,{4-1 /\x });

\draw[thick, purple!50!black, smooth, samples =20, domain=.208:.25] plot(\x,{1 /\x -4});
\draw[thick, purple!50!black, smooth, samples =20, domain=.2:.208] plot(\x,{5-1 /\x });

\draw[thick, purple!50!black, smooth, samples =20, domain=.172:.2] plot(\x,{1 /\x -5});
\draw[thick, purple!50!black, smooth, samples =20, domain=.167:.172] plot(\x,{6-1 /\x });

\draw[thick, purple!50!black, smooth, samples =20, domain=.147:.167] plot(\x,{1 /\x -6});
\draw[thick, purple!50!black, smooth, samples =20, domain=.143:.147] plot(\x,{7-1 /\x });

\draw[thick, purple!50!black, smooth, samples =20, domain=.128:.143] plot(\x,{1 /\x -7});
\draw[thick, purple!50!black, smooth, samples =20, domain=.125:.128] plot(\x,{8-1 /\x });

\draw[thick, purple!50!black, smooth, samples =20, domain=.114:.125] plot(\x,{1 /\x -8});
\draw[thick, purple!50!black, smooth, samples =20, domain=.111:.114] plot(\x,{9-1 /\x });

\draw[thick, purple!50!black, smooth, samples =20, domain=.102:.1112] plot(\x,{1 /\x -9});
\draw[thick, purple!50!black, smooth, samples =20, domain=.101:.102] plot(\x,{10-1 /\x });

\draw[thick, purple!50!black, smooth, samples =20, domain=.093:.1] plot(\x,{1 /\x -10});
\draw[thick, purple!50!black, smooth, samples =20, domain=.091:.093] plot(\x,{11-1 /\x });

\filldraw[purple!50!black] (0,0) rectangle (.09,.8);
\draw[thick, red, dotted](0,.21)node[left]{\tiny $1-\alpha$}--(1,.21);
\draw[thick, red, dotted](0,.79)node[left]{\tiny $\alpha$}--(1,.79);
\end{tikzpicture}}
\subfigure[$T_\alpha$]{\begin{tikzpicture}[scale=2.2]
\draw[white] (-.6,0)--(1.5,0);
\draw(0,0)node[below]{\tiny $0$}--(1,0)node[below]{\tiny $1$}--(1,1)--(0,1)node[left]{\tiny $1$}--(0,0);
 \draw[dotted] (.5,.2)--(.5,1);
  \draw[dotted] (.33,.2)--(.33,1);
  \draw[dotted] (.25,.2)--(.25,1);
 \draw[dotted] (0,0)--(1,1);
  \draw[dotted] (.555,0)--(.555,1);
  \node at (.56,-.15){\small $\frac1{1+\alpha}$};
 \draw[dotted] (.357,.2)--(.357,1);

 \draw[dotted] (.263,.2)--(.263,1);

\draw[thick, purple!50!black, smooth, samples =20, domain=.555:1] plot(\x,{2-1 / \x });
\draw[thick, purple!50!black, smooth, samples =20, domain=.5:.555] plot(\x,{1 / \x -1});

\draw[thick, purple!50!black, smooth, samples =20, domain=.357:.5] plot(\x,{3-1 /\x });
\draw[thick, purple!50!black, smooth, samples =20, domain=.334:.357] plot(\x,{1 /\x-2});

\draw[thick, purple!50!black, smooth, samples =20, domain=.263:.332] plot(\x,{4-1 /\x });
\draw[thick, purple!50!black, smooth, samples =20, domain=.25:.263] plot(\x,{1 /\x -3});

\draw[thick, purple!50!black, smooth, samples =20, domain=.21:.25] plot(\x,{5-1 /\x });

\draw[thick, red, dotted](0,.21)node[left]{\tiny $1-\alpha$}--(1,.21);
\draw[thick, red, dotted](0,.79)node[left]{\tiny $\alpha$}--(1,.79);
\draw[thick, red, dotted] (.21,0)node[below]{\tiny $1-\alpha$} -- (.21,1);
\end{tikzpicture}}
\caption{The Gauss map $G$ and the flipped map $R=1-G$ in (a) and (d). The folded $\alpha$-continued fraction map $\hat S_\alpha$ and the flipped $\alpha$-continued fraction map $T_\alpha$ for $\alpha < \frac12$ in (b) and (c) and for $\alpha > \frac12$ in (e) and (f).}
\label{f:off}
\end{figure}

\vskip .2cm
The first main result of this article is on the matching behaviour of the family $T_\alpha$.

\begin{thm}\label{t:main2}
The set of parameters $\alpha \in (0,1)$ for which the transformation $T_\alpha$ does not have matching is a Lebesgue null set of full Hausdorff dimension.
\end{thm}

We also give an explicit description of the matching intervals by relating them to the matching intervals of Nakada's $\alpha$-continued fraction transformations. The matching behaviour allows us to construct a planar version of the natural extension for $\alpha \in (0, \frac12 \sqrt{2})$ leading to the following result.

\begin{thm}\label{t:main1}
Let $0 \leq \alpha \leq \frac{1}{2}\sqrt{2}$, let $\mathcal{B}_\alpha$ be the Borel $\sigma$-algebra on $I_\alpha$ and let $g:=\frac{\sqrt{5}-1}{2}$. The absolutely continuous measure $\mu_\alpha$ on $(I_\alpha, \mathcal{B}_\alpha)$ with density
\begingroup
\renewcommand*{\arraystretch}{2.2}
\[
f_\alpha (x)  \hspace{-0.15em} = \hspace{-0.2em} \left\{ \hspace{-0.7em}
\begin{array}{l l}
\frac{1}{x} \mathbf{1}_{[\alpha, \frac{\alpha}{1-\alpha}]} (x) + \frac1{1+x} \mathbf{1}_{[\frac{\alpha}{1-\alpha},1-\alpha]} (x)  + \frac{2}{1-x^2} \mathbf{1}_{[1-\alpha,1]}(x),  & \text{for } \alpha\in[0,\frac{1}{2}),\\

\frac{1}{1-x} \mathbf 1_{[1-\alpha, \alpha]}  (x) + \frac1{x(1-x)} \mathbf 1_{[\alpha, \frac{1-\alpha}{\alpha} ]} (x)  + \frac{x^2+1}{x(1-x^2)} \mathbf 1_{[\frac{1-\alpha}{\alpha},1]} (x),  & \text{for } \alpha\in [\frac{1}{2}, g),\\

(\frac{1}{1-x} + \frac{1}{x+\frac{1}{g-1}})\mathbf 1_{[1-\alpha, \frac{2\alpha-1}{\alpha}]}  (x) + \frac1{1-x} \mathbf 1_{[\frac{2\alpha-1}{\alpha}, \alpha ]}  (x) +& \\
+ (\frac{1}{1-x} + \frac{1}{x} -\frac{1}{x+\frac{1}{g}}) \mathbf 1_{[\alpha,\frac{2\alpha-1}{1-\alpha}]}  (x) +  \frac{x^2+1}{x(1-x^2)} \mathbf 1_{[\frac{2\alpha-1}{1-\alpha},1]}  (x),  & \text{for } \alpha\in [g, \frac{2}{3}),\\
(\frac{1}{1-x} + \frac{1}{x+\frac{1}{g-1}})\mathbf 1_{[1-\alpha, \frac{2\alpha-1}{\alpha}]}  (x)  + \frac1{1-x} \mathbf 1_{[\frac{2\alpha-1}{\alpha}, \alpha ]}  (x)  + &\\
+ (\frac{1}{1-x} + \frac{1}{x} -\frac{1}{x+\frac{1}{g}}) \mathbf 1_{[\alpha,\frac{1-\alpha}{2\alpha-1}]} (x)+ &\\
 +(\frac{1}{1-x}+\frac{1}{x+1} -\frac{1}{x+ \frac1g} + \frac{1}{x} - \frac{1}{x+\frac{1}{g+1}}) \mathbf 1_{[\frac{1-\alpha}{2\alpha-1},1]} (x),  &  \text{for } \alpha\in [\frac{2}{3}, \frac{1}{2}\sqrt{2}),\\
\end{array}\right.
\]
\endgroup
is the unique (up to scalar multiplication) $\sigma$-finite, infinite absolutely continuous invariant measure for $T_\alpha$. Furthermore, for $\alpha\in(0,g]$ the Krengel entropy equals $\frac{\pi^2}{6}$. For $\alpha\in(0,\frac{1}{2}\sqrt{2})$ the wandering rate is given by  $w_n(T_\alpha) \sim \log n$ and the return sequence by $a_n(T_\alpha) \sim \frac{n}{\log n}$.
\end{thm}

\vskip .2cm
The paper is organised as follows. In the next section we give some preliminaries on continued fractions and explain how the maps $T_\alpha$ can be used to generate them for numbers in the interval $I_\alpha$. We also prove that the maps $T_\alpha$ fall into the family of what are called AFN-maps in \cite{Zwe00}. In the third section we study the phenomenon of matching, leading to Theorem~\ref{t:main2}, and we give an explicit description of the matching intervals. The fourth section is devoted to defining a planar natural extension for the maps $T_\alpha$ for $\alpha < \frac12  \sqrt{2}$. This is then used to obtain the invariant densities appearing in Theorem~\ref{t:main1}. In the last section we compute the Krengel entropy, wandering rate and return sequence for $T_\alpha$, giving the last part of Theorem~\ref{t:main1}.

\section{Preliminaries}\label{s:regular}

\subsection{Semi-regular continued fraction expansions}
In 1913, Perron introduced the notion of {\em semi-regular} continued fraction expansions, which are finite or infinite expressions for real numbers of the following form:
\[ x = d_0 + \cfrac{\epsilon_0}{d_1 + \cfrac{\epsilon_1}{d_2 + \ddots + \cfrac{\epsilon_{n-1}}{d_n + \ddots}}},\]
where $d_0 \in \mathbb Z$ and for each $n \ge 1$, $\epsilon_{n-1} \in \{-1,1\}$, $d_n \in \mathbb N$ and $d_n+\epsilon_n \ge 1$; see for example \cite{Per13}. We denote the semi-regular continued fraction expansion of a number $x$ by
\[ x = [d_0; \epsilon_0/d_1, \epsilon_1/d_2, \epsilon_2/d_3, \ldots].\]

\vskip .2cm
The maps $T_\alpha$ generate semi-regular continued fraction expansions of real numbers by iteration. Define for any $\alpha \in (0, 1)$ and any $x \in I_\alpha$ the partial quotients $d_k = d_k(x) = d_1(T^{k-1}_{\alpha}(x))$ and the signs $\epsilon_k = \epsilon_k(x) = \epsilon_1(T^{k-1}_{\alpha}(x))$ by setting
$$d_1(x):=\begin{cases}
\floor{\frac{1}{x}} ,& \text{if } x \in D_\alpha^{\mathsf{c}},\\
\floor{\frac{1}{x}}+1, & \text{otherwise};
\end{cases}
\quad \text{ and } \quad
\epsilon_1(x):=\begin{cases}
1 ,& \text{if } x \in D_\alpha^{\mathsf{c}},\\
-1, & \text{otherwise}.
\end{cases}$$
With this notation the map $T_{\alpha}$ can be written as $T_{\alpha}(x)=\epsilon_1(x) \big( \frac{1}{x} -d_1(x)\big)$, implying
\begin{equation}\label{q:fincf}
x= \frac{1}{d_1+\epsilon_1T_{\alpha}(x) }=  \cfrac{1}{d_1+	\cfrac{\epsilon_1}{d_2+ \ddots + \cfrac{\epsilon_{n-1}}{d_n+\epsilon_n T^n_{\alpha}(x) }} } .	
\end{equation}
Denote by $(p_n/q_n)_{n \geq 1}$ the sequence of convergents of such an expansion, that is,
\[ p_n/q_n = [0;1/d_1, \epsilon_1/d_2, \ldots, \epsilon_{n-1}/d_n].\]
Since we obtained $T_\alpha$ from the Gauss map, by flipping on the domain $D_\alpha$ from \eqref{q:Dalpha}, it follows from \cite[Theorem 1]{DHKM12} that for any $x \in I_\alpha$ we have: $\displaystyle \lim_{n \rightarrow \infty} \frac{p_n}{q_n} = x$. Therefore, we can write
\[
x =  \cfrac{1}{d_1+	\cfrac{\epsilon_1}{d_2+ \ddots + \cfrac{\epsilon_{n-1}}{d_n+\ddots }} } =: [0; 1/d_1, \epsilon_1/d_2, \epsilon_2/d_3, \ldots]_\alpha,
\]
which we call the {\em flipped $\alpha$-continued fraction expansion} of $x$.

\vskip .2cm
In case $\epsilon_n=1$ for all $n\ge 1$ the continued fraction expansion is called {\em regular} and we use the common notation $[a_1,a_2, a_3, \ldots]$ for them.
Regular continued fraction expansions are generated by the Gauss map $G:[0,1]\to [0,1]$ given by $G(0)=0$ and $G(x) = \frac1x \pmod 1$ if $x \neq 0$. Therefore, $G$ acts as a shift on the regular continued fraction expansions:
\[ x = [a_1, a_2, a_3, \ldots] \quad  \Rightarrow \quad G(x) = [a_2, a_3, a_4, \ldots].\]
It is well known that the regular continued fraction expansion of a number $x$ is finite if and only if $x \in \mathbb Q$. For any $x \in [0, \frac12\big]$ the following correspondence between the regular continued fraction expansions of $x$ and $1-x$ holds:
\begin{equation}\label{q:rcfalpha}
x = [a_1,a_2 , a_3, \ldots] \quad \Leftrightarrow \quad 1-x = [1, a_1-1,a_2, a_3,  \ldots].
\end{equation}
We will need this property later.

\vskip .2cm
On sequences of digits $(a_n)_{n \ge 1} \in \mathbb N^\mathbb N$ the {\em alternating ordering} is defined by setting $(a_n)_{n \ge 1} \prec (b_n)_{n \ge 1}$ if and only if for the smallest index $m \ge 1$ such that $a_m \neq b_m$ it holds that $(-1)^m a_m < (-1)^m b_m$. The same definition holds for finite strings of digits of the same length. The Gauss map preserves the alternating ordering, i.e.,
\[ (a_n)_{n \ge 1} \prec (b_n)_{n \ge 1} \quad \Leftrightarrow \quad [a_1, a_2, a_3, \ldots] < [b_1, b_2, b_3, \ldots].\]
The next proposition will be needed in the following section.
\begin{prop}\label{p:rational}
Let $\alpha \in (0,1)$ and $x \in I_\alpha$ be given. Then $x \in \mathbb Q$ if and only if there is an $N \ge 0$ such that $T_\alpha^N(x)=1$.
\end{prop}

\begin{proof}
If there is an $N \ge 0$ such that $T_\alpha^N(x)=1$, then it follows immediately from \eqref{q:fincf} that $x \in \mathbb Q$. Suppose $x \in \mathbb Q$. Note that $T_\alpha^n (x) \in \mathbb Q \cap I_\alpha$ for all $n \ge 0$ and write $T_\alpha^n(x) = \frac{s_n}{t_n}$ with $s_n, t_n \in \mathbb N$ and $t_n$ as small as possible. Assume for a contradiction that $T_\alpha^n(x) \neq 1$ for all $n \ge 1$. Then $s_n < t_n$ and since either $T_\alpha^{n+1}(x) = \frac{t_n-ks_n}{s_n}$ or $T_{\alpha}^{n+1} (x) = \frac{(k+1)s_n-t_n}{s_n}$, we get $0 < t_{n+1} < t_n$. This gives a contradiction.
\end{proof}

\subsection{AFN-maps}
We start our investigation into the dynamical properties of the maps $T_\alpha$ by showing that they fall into the category of AFN-maps considered in \cite{Zwe00}. Let $\lambda$ denote the one-dimensional Lebesgue measure and let $X$ be a finite union of bounded intervals. A map $T: X \to X$ is called an {\em AFN-map} if there is a finite partition $\mathcal{P}$ of $X$ consisting of non-empty, open intervals $I_i$, such that the restriction $T\mid_{I_i}$ is continuous, strictly monotone and twice differentiable. Moreover, $T$ has to satisfy the following three properties: \label{assu}
\begin{itemize}
\item[\textbf{(A)}] Adler's condition: $\frac{T''}{(T')^2}$ is bounded on $\cup_i I_i$;
\item[\textbf{(F)}] The finite image condition: $T(\mathcal{P}):=\{T(I_i): I_i \in \mathcal{P}\}$ is finite;
\item[\textbf{(N)}] The repelling indifferent fixed point condition: there exists a finite set $\mathcal{Z} \subseteq \mathcal{P}$, such that each $Z_i \in \mathcal{Z}$ has an indifferent fixed point $x_{Z_i}$, that is,
$$\lim_{x \rightarrow x_{Z_i}, x \in Z_i} T_\alpha(x)= x_{Z_i} \quad \text{ and } \quad \lim_{x \rightarrow x_{Z_i}, x \in Z_i} T' (x)=1,$$ $T'$ decreases on $(-\infty, x_{Z_i}) \cap Z_i$ and increases on $(x_{Z_i}, \infty) \cap Z_i$. Lastly, $T$ is assumed to be uniformly expanding on sets bounded away from $\{x_{Z_i}: Z_i \in \mathcal{Z}\}$.
\end{itemize}

\vskip .2cm
For the maps $T_{\alpha}$ we can take $\mathcal P$ to be the collection of intervals of monotonicity (or cylinder sets) of $T_\alpha$, defined for each $\epsilon \in \{-1,1\}$ and $d \ge 1$ by
\begin{equation}\label{q:cylinders}
\Delta(\epsilon, d) = int \{ x \in I_\alpha \, : \, \epsilon_1(x)=\epsilon \text{ and } d_1(x)=d\},
\end{equation}
where we use $int$ to denote the interior of the set.

\begin{lem}\label{l:afn}
For each $\alpha \in ( 0,1)$ the map $T_\alpha$ is an AFN-map.
\end{lem}

\begin{proof}
Let $\mathcal P = \{ \Delta (\epsilon, d) \}$. Then $T_\alpha$ is continuous, strictly monotone and twice differentiable on each of the intervals in $\mathcal P$. We check the three other conditions. For \textbf{(A)} note that $T_\alpha'(x)= \pm \frac1{x^2}$, so that $\frac{T_\alpha''(x)}{(T_\alpha'(x))^2} = \pm \frac{2x^4}{x^3} = \pm 2x\le 2$ for any $x$ for which $T_\alpha'$ is defined. Also, for any $J \in \mathcal P$ we have
\[ T_\alpha ( J ) \in \big\{ (\alpha,1), (1-\alpha,1), (\alpha, T_\alpha(\alpha)), (1-\alpha, T_\alpha(1-\alpha)), (T_\alpha(\alpha),1), (T_\alpha(1-\alpha),1)\big\},\]
giving \textbf{(F)}. Finally, $T_\alpha$ has only 1 as an indifferent fixed point. Since $T'_{\alpha}(x)=1/x^2>1$ for any $x \in I_\alpha \setminus \{1\}$ where $T_\alpha'(x)$ is defined, we see that $T'_\alpha$ decreases near 1 and also \textbf{(N)} holds.
\end{proof}

Using \cite[Theorem A]{Zwe00} we then obtain the following result.
\begin{prop}\label{t:acim}
For each $\alpha \in \big(0,1)$ there exists a unique absolutely continuous, infinite, $\sigma$-finite $T_\alpha$-invariant measure $\mu_\alpha$ that is ergodic and conservative for $T_\alpha$.
\end{prop}

\begin{proof}
Since $T_\alpha$ is an AFN-map, \cite[Theorem A]{Zwe00} immediately implies that there are finitely many disjoint open sets $X_1, \ldots, X_N \subseteq I_\alpha$, such that $T_\alpha(X_i) = X_i \pmod \lambda$ and $T|_{X_i}$ is conservative and ergodic with respect to $\lambda$. Each $X_i$ is a finite union of open intervals and supports a unique (up to a constant factor) absolutely continuous $T_\alpha$-invariant measure. Moreover, this invariant measure is infinite if and only if $X_i$ contains an interval $(1-\delta, 1)$ for some $\delta>0$. Since each open interval contains a rational point in its interior, Proposition~\ref{p:rational} together with the forward invariance of the sets $X_i$ implies that there can only be one set $X_i$ and that this set contains an interval of the form $(1-\delta, 1)$. Hence, there is a unique (up to a constant factor) absolutely continuous invariant measure $\mu_\alpha$ that is infinite, $\sigma$-finite, ergodic and conservative for $T_\alpha$.
\end{proof}

From Proposition~\ref{t:acim} and \cite[Theorem 1]{Zwe00} it follows that each map $T_{\alpha}$ is {\em pointwise dual-ergodic}, i.e., there are positive constants $a_n(T_{\alpha})$, $n \geq 1$, such that for each $f \in L^1(I_\alpha, \mathcal B_{\alpha}, \mu_\alpha)$, where $\mathcal B_{\alpha}$ is the Borel $\sigma-$algebra on $I_{\alpha}$,
\begin{equation}\label{e:pde}
\lim_{n \to \infty} \frac{1}{a_n(T_{\alpha})} \sum_{k=0}^{n-1} \mathbf{P}_{\alpha}^k f = \int_{I_\alpha} f \, d \mu_{\alpha} \quad \mu_{\alpha}\text{-a.e.},
\end{equation}
where $\mathbf{P}_{\alpha}$ denotes the {\em transfer operator} of the map $T_\alpha$, defined by
\begin{equation}\nonumber
\int_A \mathbf{P}_{\alpha} f \, d\mu_\alpha = \int_{T_\alpha^{-1}(A)} f \, d\mu_\alpha \quad \text{ for all } f \in L^1(I_\alpha, \mathcal B_{\alpha}, \mu_\alpha) \text{ and } A \in \mathcal B_{\alpha}.
\end{equation}
The sequence $(a_n(T_\alpha))_{n \ge 1}$ is called the {\em return sequence} of $T_\alpha$ and will be given for $\alpha \in \big(0, \frac12 \sqrt 2 \big)$ in Section~\ref{s:iso}.

\section{Matching almost everywhere}\label{s:matchingae}
In this section we prove that matching holds for almost every $\alpha\in(0,1)$. The discontinuity points of the map $T_\alpha$ are of the form $\frac1{k+\alpha}$ for some positive integer $k$. For any such point,
\[ c^-=\lim_{x \uparrow \frac1{k+\alpha}} T_\alpha(x) = \alpha \quad \text{and} \quad c^+=\lim_{x \downarrow \frac1{k+\alpha}} T_\alpha(x)  = -(k+\alpha) + k+1=1-\alpha.\]
Recall the definition of matching from equation~(\ref{q:matchingdef}): matching for $T_{\alpha}$ holds if there exist non-negative integers $M, N$ such that
\begin{equation}\label{e:matching}
T_{\alpha}^M (\alpha) = T_{\alpha}^N (1-\alpha).
\end{equation}
Some authors also require the evaluation of the derivative of the iterates in the left and right limits of the critical points to coincide. In our case, we do not need this constraint, since we prove that matching is a local property.

\vskip .2cm
In the next proposition we show that the first half of the parameter space consists of a single matching interval.

\begin{prop}\label{p:matching1}
For $\alpha \in \big(0, \frac12\big)$ it holds that $T_\alpha(\alpha) = T^2_\alpha(1-\alpha)$.
\end{prop}

\begin{proof}
Fix $\alpha = [a_1,a_2, \ldots] \in \big(0, \frac12 \big)$. First note that $\frac12 < 1-\alpha < \frac1{1+\alpha}$, so that by \eqref{q:rcfalpha} we obtain that
\begin{equation}\nonumber
T_\alpha(1-\alpha) = G(1-\alpha) = \frac{\alpha}{1-\alpha}  = [a_1-1,a_2,a_3, \ldots].
\end{equation}
Hence
\[ \frac{1}{a_1} < \alpha < \frac{1}{a_1-1+\alpha } \Leftrightarrow \frac{1}{a_1-1} < \frac{\alpha}{1-\alpha} < \frac{1}{a_1-2+\alpha },\]
which gives that either $\alpha$ and $T_\alpha(1-\alpha)$ are both in $D_\alpha$ or in $D_\alpha^{\mathsf{c}}$. In both cases,
\[T_\alpha (\alpha) = T_\alpha^2(1-\alpha). \qedhere \]
\end{proof}

For $\alpha > \frac12$ the situation is much more complicated. One explanation for this difference comes from two operations that convert one semi-regular continued fraction expansion of a number into another: singularisation and insertion. Both operations were introduced in \cite{Per13} and later appeared in many other places in the literature, see e.g.~\cite{Kra91,DK00,HK02,Sch04,DHKM12}. {\em Singularisation} deletes one of the convergents $\frac{p_n}{q_n}$ from the sequence while altering the ones before and after; {\em insertion} inserts the mediant $\frac{p_n+p_{n+1}}{q_n+q_{n+1}}$ of $\frac{p_n}{q_n}$ and $\frac{p_{n+1}}{q_{n+1}}$ into the sequence. It follows from \cite[Section 2.1]{DHKM12} that for $\alpha < \frac12$ the flipped $\alpha$-continued fraction expansions of numbers in $I_{\alpha}$ can be obtained from their regular continued fraction expansions by insertions only, while for $\alpha>\frac{1}{2}$ one needs singularisations as well.

\vskip .2cm
Define the map $R:[0,1]\to [0,1]$ by $R(x) = 1-G(x)$, see Figure~\ref{f:off}(d). Before we prove that matching holds Lebesgue almost everywhere, we describe the effect of $R$ on the regular continued fraction expansions of numbers in $(0,1)$.

\begin{lem}\label{l:reverserenyi}
Let $x \in (0,1)$ have regular continued fraction expansion $x=[x_1, x_2, x_3, \ldots]$. Then for each $j \ge 1$,
\[ R^{x_2 + x_4 + \cdots + x_{2j}} (x) = [ x_{2j+1}+1 , x_{2j+2}, x_{2j+3}  , \ldots]\]
and if $0<\ell < a_{2j}$, then
\[ R^{x_2 + x_4 + \cdots + x_{2j-2}+\ell} (x) = [ 1, x_{2j}-\ell ,x_{2j+1}, x_{2j+2}  , \ldots].\]
\end{lem}
\begin{proof}
By \eqref{q:rcfalpha} it holds that
\[
R(x) = 1-G(x) = 1-[ x_2, x_3, x_4, \ldots]
= \begin{cases}
[ x_3+1, x_4, x_5, \ldots] = R^{x_2}(x), & \text{if } x_2=1,\\
[ 1, x_2-1, x_3, x_4,\ldots], & \text{if } x_2 >1.
\end{cases}
\]
The statement then easily follows by induction.
\end{proof}

\begin{nrem}\label{r:parity}
{\rm The previous lemma implies that $R$ preserves the parity of the regular continued fraction digits. More precisely, if $x \in (0,1)$, then (except for possibly the first two digits) the regular continued fraction expansion of $R(x)$ has regular continued fraction digits of $x$ with even indices in even positions and regular continued fraction digits with odd indices in odd positions. }
\end{nrem}

The map $T_\alpha$ equals the map $R$ on $D_\alpha$ and $G$ on $D_\alpha^{\mathsf{c}}$. The next lemma specifies the times $n$ at which the orbit of $\alpha$ (or $1-\alpha$) can enter $D_\alpha^{\mathsf{c}}$ for the first time.

\begin{lem}\label{l:entrytimes}
Let $\alpha=[1, a_1,a_2,a_3,\ldots] \in \big( \frac12, 1 \big)$. If $m:= \min \{ i \ge 0 \, : \, T_\alpha^i(\alpha) \in D_\alpha^{\mathsf{c}} \}$ exists, then $m=a_1+a_3+ \cdots + a_{2j+1}-1$ where $j$ is the unique integer such that
\[ a_1 + a_3 + \cdots + a_{2j-1}-1 < m \le a_1 +a_3 + \cdots + a_{2j+1}-1.\]
Similarly, if $k:=\min \{ i \ge 0 \, : \, T_\alpha^i(1-\alpha) \in D_\alpha^{\mathsf{c}} \}$ exists, then $k=a_2+a_4+\cdots + a_{2j}-1$ where $j$ is the unique integer such that
\[ a_2+ a_4+ \cdots + a_{2j-2}-1 < k \le a_2+a_4+ \cdots + a_{2j}-1.\]
\end{lem}

\begin{proof}
For the first statement, by the definition of $m$ we know that $T^i_{\alpha}(\alpha)= R^i(\alpha)$ for all $i \leq m$. From Lemma~\ref{l:reverserenyi} it then follows that if $m= a_1 + a_3 + \cdots + a_{2j-1}$, then
\[ T_\alpha^m (\alpha) = [ a_{2j}+1 ,a_{2j+1}, a_{2j+2}  , \ldots],\]
and if $m= a_1 + a_3 + \cdots + a_{2j-1}+\ell$ for some $0 < \ell \leq a_{2j+1}-1$, then
\[ T_{\alpha}^m (\alpha) = [ 1, a_{2j+1}-\ell ,a_{2j+2}, a_{2j+3}  , \ldots].\]
Recall that $D_\alpha^{\mathsf{c}} = \bigcup_d \Delta(1, d)$. The right boundary point of any cylinder $\Delta(1, d)=(\frac{1}{d+1}, \frac{1}{d+\alpha})$ has regular continued fraction expansion $[ d, 1, a_1 , a_2, \ldots]$. Since the regular continued fraction expansion of the left boundary point also starts with the digits $d,1$, any $x \in D_\alpha^{\mathsf{c}}$ has a regular continued fraction expansion of the form $[x_1,1,x_3, \ldots]$. In particular this holds for $T_\alpha^m(\alpha)$, which implies that either $a_{2j+1}=1$ or $\ell=a_{2j+1}-1$. In both cases, $ m= a_1+a_3+ \cdots +a_{2j+1}-1$. For the second part of the lemma, recall from \eqref{q:rcfalpha} that $1-\alpha = [a_1+1, a_2, a_3, \ldots]$. The proof of the second part then goes along the same lines as above.
\end{proof}

\vskip .2cm
Recall from the introduction the definition of matching intervals as the maximal parameter intervals on which the matching exponents $M,N$ from~\eqref{q:matchingdef} are constant. We can obtain a complete description of the matching intervals by relating them to the matching intervals of Nakada's $\alpha$-continued fraction maps from \eqref{q:nakada}. First we recall some notation and results on matching for the maps from \eqref{q:nakada}. Any rational number $a \in \mathbb{Q} \cap (0,1)$ has two regular continued fraction expansions:
\[a=[a_1, \ldots, a_n]= [a_1, \ldots, a_n-1,1], \quad a_n \geq 2.\]
The {\em quadratic interval} $I_a$ associated to $a$ is the interval with endpoints
\[[\overline{a_1, \ldots, a_n}] \quad \text{ and } \quad [\overline{a_1, \ldots, a_n-1,1}].\]
The quadratic interval $I_1$ is defined separately by $I_1 = (g,1)$, where $g = \frac{\sqrt 5-1}{2}$. A quadratic interval $I_a$ is called \textit{maximal} if it is not properly contained in any other quadratic interval. By \cite[Theorem 1.3]{CT12} maximal intervals correspond to matching intervals for Nakada's $\alpha$-continued fraction maps.

\vskip .2cm
Let $\mathcal R = \{ a \in \mathbb Q \cap (0,1] \, : \, I_a \text{ is maximal} \}$ and $a = [a_1, \ldots, a_n] \in \mathcal R$ with $a_n \ge 2$. The map $x \mapsto \frac1{1+x}$ is the inverse of the right most branch of the Gauss map. Therefore, $\frac1{1+a} = [1,a_1, a_2, \ldots, a_n-1,1] = [1,a_1, a_2, \ldots, a_n]$. Write
\[ \begin{split} J_a^L =\ & \Big( [1,\overline{a_1, a_2, \ldots , a_n-1,1}], [1,a_1, a_2, \ldots, a_n-1,1] \Big),\\
J_a^R =\ & \Big([1,a_1, a_2, \ldots, a_n], [1,\overline{a_1, a_2, \ldots , a_n}] \Big),
\end{split}\]
if $n$ is odd and
\[ \begin{split} J_a^L =\ & \Big( [1,\overline{a_1, a_2, \ldots , a_n}], [1,a_1, a_2, \ldots, a_n] \Big),\\
J_a^R =\ & \Big([1,a_1, a_2, \ldots, a_n-1,1], [1,\overline{a_1, a_2, \ldots , a_n-1,1}] \Big),
\end{split}\]
if $n$ is even, so that $\frac1{1+I_a} = J_a^L \cup J_a^R \cup \big\{ \frac1{1+a} \big\}$. Finally, let
\begin{equation}\label{q:oddn}
M= a_1+ a_3 + \cdots + a_n \quad \text{ and } \quad N= a_2+ a_4 + \cdots + a_{n-1}+2
\end{equation}
if $n$ is odd and
\begin{equation}\label{q:neven}
M= a_1+ a_3 + \cdots + a_{n-1}+1 \quad \text{ and } \quad N= a_2+ a_4 + \cdots + a_n+1
\end{equation}
if $n$ is even. The next theorem states that the intervals $J_a^L$ and $J_a^R$ are matching intervals for the flipped $\alpha$-continued fraction maps with matching exponents that depend on $M$ and $N$.

\begin{thm}\label{t:acf}
Let $a \in \mathcal R$ and let $M$ and $N$ be as in \eqref{q:oddn} and \eqref{q:neven}. For each $\alpha \in J_a^L$ the map $T_\alpha$ satisfies $T_\alpha^M(\alpha) = T_\alpha^N(1-\alpha)$ and for each $\alpha \in J_a^R$ the map $T_\alpha$ satisfies $T_\alpha^{M+1}(\alpha) = T_\alpha^{N-1}(1-\alpha)$.
\end{thm}

\begin{proof}
First we consider the special maximal quadratic interval $I_1 = (g,1)$ separately, since $J_1^L = \emptyset$ and $J_1^R = \big( \frac12, g\big)$. Let $\alpha \in J_1^R$. Then $\alpha = [1,1,a_2, a_3, \ldots]$ and $1-\alpha = [2,a_2,a_3, \ldots]$. Note that $M=1$ and $N=2$. From $\alpha < g$ it follows that $\alpha^2+\alpha-1<0$. This implies that $1-\alpha > \frac1{2+\alpha}$, so that $T^{N-1}_\alpha(1-\alpha) =T_\alpha(1-\alpha) = R(1-\alpha)$. It also implies that $\frac12 < \alpha < \frac1{1+\alpha}$ and that
\[ T_\alpha(\alpha) = G(\alpha) = \frac1{\alpha}-1 > \frac1{1+\alpha},\]
so that $T^{M+1}_\alpha(\alpha) = T^2_\alpha(\alpha) = R \circ G (\alpha)$, which by Lemma~\ref{l:reverserenyi} equals $T^{N-1}_\alpha(1-\alpha)$.

\vskip .2cm
Fix $a \in \mathcal R \setminus \{1\}$ and write $a=[a_1,a_2, \ldots, a_n]=[a_1,a_2, \ldots, a_n-1,1]$ for its regular continued fraction expansions. We only prove the statement for $J_a^L$, since the proof for $J_a^R$ is similar. Assume without loss of generality that $n$ is odd. The proof is analogous for $n$ even and the parity is fixed only to determine the endpoints of $J_a^L$. We start by proving that matching cannot occur for indices smaller than $M$ and $N$.

\vskip .2cm
Write $\mathbf a =a_1, a_2,  \ldots ,  a_{n-1}, a_n-1, 1 \in \mathbb N^{n+1}$ and let $\alpha \in J_a^L = ([1,\overline{\mathbf a}], [1,\mathbf a]) $. Then there is some finite or infinite string of positive integers $\mathbf w=a_{n+2}, a_{n+3}, \ldots$, such that $\alpha=[1, \mathbf a , \mathbf w]$. The assumption that $n$ is odd together with the fact that the Gauss map preserves the alternating ordering imply that
\begin{equation}\label{e:conditionw}
\overline{\mathbf a} \succ \mathbf w.
\end{equation}
Assume that $m = \min \{ i \ge 0 \, : \, T^i_{\alpha}(\alpha) \in D_\alpha^{\mathsf{c}}\}$ exists. By Lemma~\ref{l:entrytimes} there is a $j$, such that $m = a_1 + a_3 + \cdots + a_{2j-1}-1$. Assume that $2j-1 < n$. Since $a \in \mathcal R$, the result from \cite[Proposition 4.5.2]{CT12} implies that for any two non-empty strings $\mathbf u$ and $\mathbf v$ such that $\mathbf a=\mathbf u \mathbf v$, the inequality
\begin{equation}\label{e:maximalitystring}
\mathbf v \succ \mathbf u \mathbf v
\end{equation}
holds. By the definition of $m$ it holds that $T^m_\alpha (\alpha) \in D_\alpha^{\mathsf{c}}$. So, using Lemma~\ref{l:reverserenyi} we obtain that
\[ [1, a_{2j}, a_{2j+1}, \ldots] = G(T^m_\alpha (\alpha)) = T^{m+1}_\alpha(\alpha)> \alpha = [1,\mathbf a,\mathbf w].\]
Thus, if we take $\mathbf v=a_{2j}, a_{2j+1}, \ldots, a_n-1, 1$ and $\mathbf u=a_1, a_2, \ldots, a_{2j-1}$, then we find $\mathbf v \preceq \mathbf u \mathbf v$, which contradicts \eqref{e:maximalitystring}. Hence, if $m$ exists, then $m \ge M-1$. In a similar way we can deduce that if $k = \min \{ i \ge 0 \, : \, T^i_\alpha (1-\alpha ) \in D_\alpha^{\mathsf{c}} \}$ exists, then $k \ge N-2$.

\vskip .2cm
Now assume that there exist $\ell < M-1$ and $i < N-2$, such that
\begin{equation}\label{q:matchingbefore}
T_\alpha^\ell(\alpha) = R^\ell (\alpha) =  R^i(1-\alpha) = T_\alpha^i (1-\alpha).
\end{equation}
Recall from Remark~\ref{r:parity} that $R$ preserves the parity of the regular continued fraction digits. Since $\alpha = [1,a_1,a_2, \ldots]$ and $1-\alpha = [a_1+1,a_2,a_3, \ldots]$, the assumption \eqref{q:matchingbefore} then implies the existence of an even index $2 \le j \le n-1 $ and an odd index $1 \le \ell < n$, such that
\[ a_j, a_{j+1}, a_{j+2} , \ldots = a_\ell, a_{\ell+1}, a_{\ell +2}, \ldots.\]
This implies that $a$ has an ultimately periodic regular continued fraction expansion, which contradicts the fact that $a \in \mathbb Q$. Hence, matching cannot occur with indices $\ell < M-1$ and $i < N-2$.

\vskip .2cm
Next consider $T^M_\alpha (\alpha)$ and $T^{N-2}_\alpha (1-\alpha)$. From Lemma~\ref{l:reverserenyi} we get that
\[ T^{M-1}_\alpha(\alpha) = R^{a_1 + a_3 + \cdots + a_n-1}(\alpha) = [2,\mathbf w].\]
From $\alpha = [1,\mathbf a, \mathbf w] > g$, it follows that $G(\alpha) = [\mathbf a , \mathbf w] < g$. Combining this with the fact that the property from \eqref{e:conditionw} implies $\mathbf w \prec \mathbf a \mathbf w$ gives $\mathbf w \prec \mathbf a \mathbf w \prec 1 \mathbf a \mathbf w $. Hence, $T^{M-1}_\alpha(\alpha)  > [2, 1, \mathbf a, \mathbf w] = \frac1{2+\alpha}$. This implies that $T^M_\alpha(\alpha) = R (T^{M-1}_\alpha(\alpha)) = R([2, \mathbf w])$. For $1-\alpha =[a_1+1, a_2, a_3, \ldots, a_{n-1}, a_n-1, 1, \mathbf w]$ we get from Lemma~\ref{l:reverserenyi} that
\[ T_\alpha^{N-2}(1-\alpha) = R^{N-2}(1-\alpha) = [a_n, 1, \mathbf w].\]
Again using that $\mathbf w \prec \mathbf a \mathbf w$ gives $T_\alpha^{N-2}(1-\alpha) \in \Delta(1,a_n) = ([a_n,1],[a_n,1,\mathbf a, \mathbf w])$. Since $\alpha > g$, it follows that $T^N_\alpha (1-\alpha) = R \circ G (T_\alpha^{N-2}(1-\alpha))$. Then, again by using Lemma~\ref{l:reverserenyi}, we obtain
\[ T^N_\alpha (1-\alpha) = R([1, \mathbf w]) = R([2,\mathbf w]) = T^M_\alpha (\alpha).\]

\vskip .2cm
For $\alpha \in J_a^R$, one can show similarly that $T^{M-1}_\alpha(\alpha) = [1,1,\mathbf w] \in \Delta(1,1)$. Since $\alpha > g$, this gives $T^{M+1}_\alpha(\alpha) = R \circ G (T^{M-1}_\alpha(\alpha)) = R([1,\mathbf w])$.
On the other hand, $T^{N-2}_\alpha(1-\alpha) = R^{N-2}_\alpha(1-\alpha) = [a_n+1, \mathbf w] > \frac1{a_n+1+\alpha}$. So,
\[ T^{N-1}_\alpha(1-\alpha) = R^{N-1}_\alpha (1-\alpha) = R ([a_n+1, \mathbf w] ) = R([1, \mathbf w]) = T^{M+1}_\alpha (\alpha). \qedhere \]
\end{proof}

From this theorem we obtain the result from Theorem~\ref{t:main2} on the size of the set of non-matching parameters. We use $\dim_H(A)$ to denote the Hausdorff dimension of a set $A$ and let $\mathcal{E}$ denote the non-matching set, that is,
\[\mathcal{E} = \{ \alpha \in (0,1) \, : \, T_\alpha \text{ does not have the matching property} \}.\]

\begin{proof}[Proof of Theorem~\ref{t:main2}]
We use known results on the exceptional set $\mathcal N$ of non-matching parameters for Nakada's $\alpha$-continued fraction maps from \eqref{q:nakada}. It is proven in \cite{CT12} and \cite{KSS12} that $\lambda(\mathcal N)=0$ and in \cite[Theorem 1.2]{CT12} that $\dim_H(\mathcal N)=1$. Since the bi-Lipschitz map $x \mapsto \frac{1}{1+x}$ on $(0,1)$ preserves Lebesgue null sets and Hausdorff dimension, the same properties hold for the set $E:=\frac1{1+\mathcal N}$. Note that $T_\alpha$ has matching for all $\alpha \in E^{\mathsf{c}}$, since according to Theorem~\ref{t:acf} either $\alpha$ is in a matching interval or it is of the form $\frac1{1+a}$ for some rational number $a$ and then both $\alpha$ and $1-\alpha$ eventually get mapped to 1. Hence, $\mathcal E \subseteq E$ and it follows that $\lambda(\mathcal E)=0$.

\vskip .2cm
Now consider $E \setminus \mathcal E$. Let $a \in \mathcal N$. By~\cite[Section 4]{KSS12}, this is equivalent to $G^n(a) \geq a$ for all $n \geq 1$. Let $\alpha:=\frac{1}{1+a}$ and write $[1,a_1,a_2, \ldots]$ for its regular continued fraction expansion. Suppose there exists a minimal $m\ge 0$ such that $T^m_{\alpha}(\alpha) \in D_\alpha^{\mathsf{c}}$. Then there exists a positive integer $d$ such that
\[\frac{1}{d+1} < T^m_{\alpha}(\alpha) < \frac{1}{d+\alpha}.\]
By Lemma~\ref{l:reverserenyi} the inequality implies in particular that for some $j > 2$
\[ [a_j, a_{j+1}, \ldots] < [a_1, a_2, \ldots],\]
i.e., $G^{j-1}(a) < a$, which contradicts the assumption on $a$. Hence, $T^k_{\alpha}(\alpha) \not \in D_\alpha^{\mathsf{c}}$ for all $k \ge 0$. Since the regular continued fraction expansion of $1-\alpha$ is given by $1-\alpha= [a_1+1, a_2, \ldots ]$, the same conclusion holds for $1-\alpha$, that is, $T^k_{\alpha}(\alpha) \not \in D_\alpha^{\mathsf{c}}$ for all $k \ge 0$. Hence, $T^k_{\alpha}(\alpha)= R^k(\alpha)$ and $T^k_{\alpha}(1-\alpha)= R^k(1-\alpha)$ for all $k$. Assume that $\alpha \not \in \mathcal E$, so there are positive integers $M, N$ such that $T_{\alpha}^M(\alpha)= T_{\alpha}^N(1-\alpha)$. By Remark~\ref{r:parity} there is an odd index $\ell \ge 1$ and an even index $ k \ge 2 $ such that
\[ a_{\ell}, a_{\ell+1}, \ldots = a_k, a_{k+1}, \ldots .\]
Therefore $\alpha$ is ultimately periodic and thus a preimage of a quadratic irrational. This implies that $\dim_H ( E \setminus \mathcal E \big)=0$ and hence $\dim_H(\mathcal E)=1$.
\end{proof}

These matching results are the main reason for the existence of the nice geometric versions of the natural extensions that we investigate in the next section.

\section{Natural extensions}\label{s:ne}
For non-invertible dynamical systems, especially for continued fraction transformations, the natural extension is a very useful tool to obtain dynamical properties of the system. Roughly speaking, the natural extension of a system is the minimal invertible dynamical system that contains the original system as a subsystem. Canonical constructions of the natural extension were first studied by Rohlin in \cite{Roh61}. Based on these results it was shown in \cite{Si88, ST91} that for infinite measure systems like $T_\alpha$ a natural extension always exists and that any two natural extensions of the same system are necessarily isomorphic. Moreover, many ergodic properties carry over from the natural extension to the original map. The amount of information on the original system that can be gained from the natural extension, depends to a large extent on the version of the natural extension one considers. For continued fraction maps, there is a canonical construction that has led to many useful observations; see for example \cite{Nak81,Kra91,KSS12,AS, Ha02}. It turns out that a similar construction also works for the family $\{T_\alpha \}_{\alpha \in (0, 1)}$.

\vskip .2cm
In this section we construct a natural extension for the system $(I_{\alpha},  \mathcal B_{\alpha}, \mu_{\alpha}, T_{\alpha})$, where $\mathcal B_{\alpha}$ is the Borel $\sigma-$algebra on $I_{\alpha}$ and $\mu_{\alpha}$ is the measure from Proposition~\ref{t:acim}. This natural extension is given by the dynamical system $(\mathcal D_\alpha, \mathcal B(\mathcal D_\alpha), \nu_\alpha , \mathcal T_\alpha)$, where $\mathcal D_\alpha$ is some domain in $\mathbb R^2$ that needs to be determined, $\mathcal B(\mathcal D_\alpha)$ is the Borel $\sigma$-algebra on $\mathcal D_\alpha$, $\nu_\alpha$ is the measure defined by
\begin{equation} \label{e:canme}
\nu_\alpha (A) =  \iint_A \frac1{(1+xy)^2} \ d\lambda^2(x,y) \quad \text{ for any } A \in \mathcal B(\mathcal D_\alpha),
\end{equation}
where $\lambda^2$ is the two-dimensional Lebesgue measure, and $\mathcal T_\alpha : \mathcal D_\alpha \to \mathcal D_\alpha$ is given by
\[ \mathcal T_\alpha (x,y) = \bigg(T_{\alpha}(x), \frac{\epsilon_1(x)}{d_1(x)+y}\bigg).\]
To prove that $(\mathcal D_\alpha, \mathcal B(\mathcal D_\alpha), \nu_\alpha, \mathcal T_\alpha)$ is the natural extension of $(I_{\alpha},  \mathcal B_{\alpha}, \mu_{\alpha}, T_{\alpha})$ we need to show that $\nu_\alpha$ is $\mathcal T_\alpha$-invariant and that all of the following properties hold $\nu_\alpha$-almost everywhere:
\begin{itemize}
\item[(ne1)] $\mathcal T_\alpha$ is invertible;
\item[(ne2)] the projection map $\pi: \mathcal D_\alpha \to I_{\alpha}$ is measurable and surjective;
\item[(ne3)] $\pi \circ \mathcal T_\alpha = T_\alpha \circ \pi$, where $\pi$ is the projection onto the first coordinate;
\item[(ne4)] $\bigvee_{n=0}^\infty \mathcal T_\alpha^n \pi^{-1}(\mathcal B_{\alpha}) = \mathcal B(\mathcal D_\alpha)$, where $\bigvee_{n=0}^\infty \mathcal T_\alpha^n \pi^{-1}(\mathcal B_{\alpha})$ is the smallest $\sigma$-algebra containing the $\sigma$-algebras $\mathcal T_\alpha^n \pi^{-1}(\mathcal B_{\alpha})$ for all $n \ge 0$.
\end{itemize}

\vskip .2cm
The shape of $\mathcal D_\alpha$ will depend on the orbits of $\alpha$ and $1-\alpha$ up to the moment of matching. As might be imagined in light of Proposition~\ref{p:matching1} and Theorem~\ref{t:acf}, the situation for $0 < \alpha < \frac12$ is simpler than for $\frac12 < \alpha < 1$. We will provide a detailed description and proof for $0 < \alpha < \frac12$ and list some analytical and numerical results for $\frac12 < \alpha < 1$.

\subsection{Natural extension for $\mathbf{T_{\alpha}}$ for $ \mathbf{\alpha < \frac12}$} We claim that for $\alpha < \frac12$ the domain of the natural extension is given by
\[ \mathcal{D}_{\alpha}:= \bigg[\alpha, \frac{\alpha}{1-\alpha}\bigg] \times [0, \infty ) \cup \bigg(\frac{\alpha}{1-\alpha}, 1-\alpha\bigg] \times [0,1] \cup \bigg(1-\alpha, 1 \bigg] \times [-1, 1], \]
see Figure~\ref{f:ne1}. Before we check (ne1)--(ne4), we introduce some notation. Partition $D_\alpha$ according to the cylinder sets of $T_{\alpha}$ described in \eqref{q:cylinders}. Let
\[ \tilde \Delta(-1, 2) = \Delta(-1,2) \times (-1,1), \quad \tilde \Delta(1,1) = \Big( \frac12, 1-\alpha \Big) \times (0,1) \cup \Big( 1-\alpha, \frac1{1+\alpha} \Big) \times (-1,1),\]
and for $(\epsilon, d) \notin \{(1,1), (-1,2)\}$,
\[ \tilde \Delta(\epsilon, d) = \begin{cases}
\Delta(\epsilon, d) \times (0,1), & \text{if } \Delta(\epsilon,d) \subseteq \big[ \frac{\alpha}{1-\alpha}, 1-\alpha \big],\\
\Delta(\epsilon, d) \times (0, \infty), & \text{if } \Delta(\epsilon,d) \subseteq \big[\alpha, \frac{\alpha}{1-\alpha}\big],
\end{cases}\]
and for the $\epsilon$ and $d$ such that $\frac{\alpha}{1-\alpha} \in \Delta(\epsilon,d)$,
\[ \tilde \Delta_L(\epsilon,d) =\Big( \Delta(\epsilon,d)\cap \Big[\alpha, \frac{\alpha}{1-\alpha} \Big] \Big) \times (0, \infty), \quad \tilde \Delta_R(\epsilon,d) =\Big( \Delta(\epsilon,d)\cap \Big[\frac{\alpha}{1-\alpha},1 \Big] \Big) \times (0, 1).\]
Due to the matching property described in Proposition~\ref{p:matching1} we have, up to a Lebesgue measure zero set,
\begin{itemize}
\item $ \mathcal{T}_{\alpha} ( \tilde \Delta(1,1))= \big( \alpha, \frac{\alpha}{1-\alpha} \big) \times \big(\frac12, \infty \big) \cup \big( \frac{\alpha}{1-\alpha}, 1 \big) \times \big(\frac12, 1 \big)$,
\item $ \bigcup_{d\ge 2} \mathcal{T}_{\alpha} ( \tilde \Delta(-1,d))= (1-\alpha, 1) \times (-1, 0)$,
\item $ \bigcup_{d \ge 1} \mathcal{T}_{\alpha} ( \tilde \Delta(1,d))= (\alpha, 1) \times (0,1)$,
\end{itemize}
where we have included the sets $\tilde \Delta_L(\epsilon, d)$ and $\tilde \Delta_R(\epsilon,d)$ in the appropriate union.
Hence, $\mathcal T_\alpha$ is Lebesgue almost everywhere invertible, which gives (ne1).

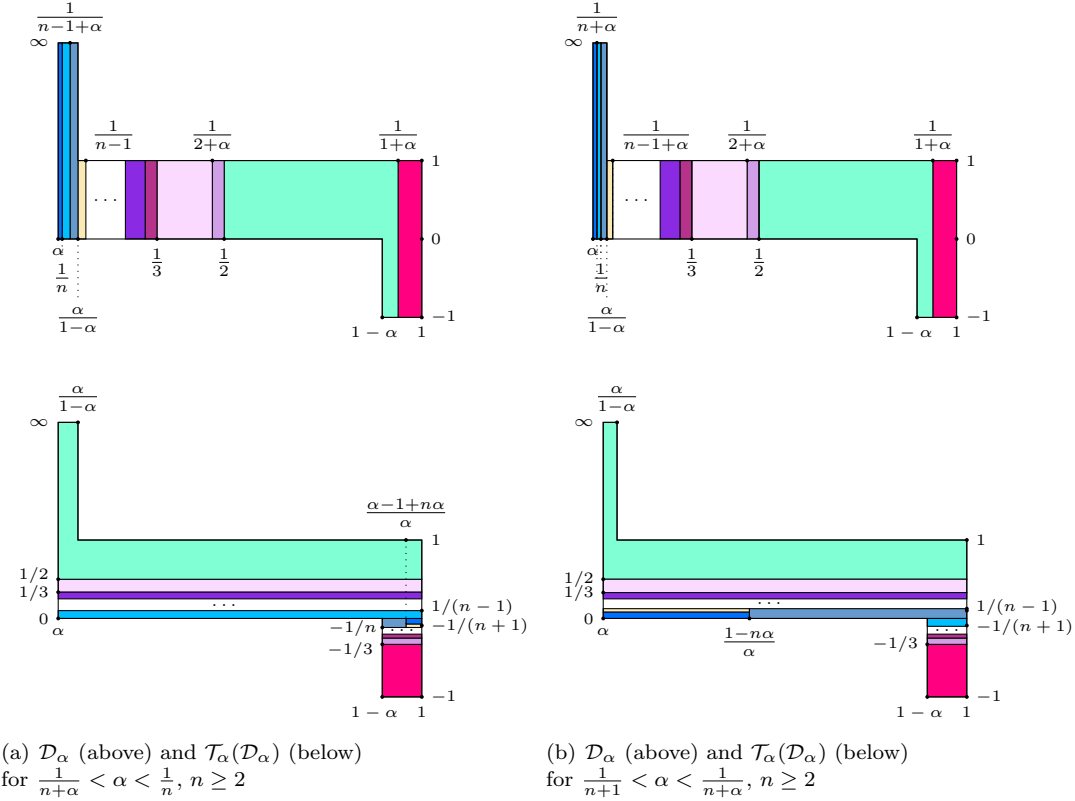
\begin{figure}
\begin{center}
\subfigure{
\begin{tikzpicture}[scale=5.2]
\draw[fill=brightpink] (.94,-.2) -- (1,-.2) -- (1,.2) -- (.94,.2)--cycle;
\draw[fill=aquamarine] (.9,-.2) -- (.94,-.2) -- (.94,.2) -- (.5,.2) -- (.5,0) -- (.9,0)--cycle;
\draw[fill=brightube] (.47,0)--(.5,0)--(.5,.2)--(.47,.2)--cycle;
\draw[fill=brilliantlavender] (.33,0) -- (.47,0) -- (.47,.2) -- (.33,.2)--cycle;
\draw[fill=fandango](.3,0)--(.33,0)--(.33,.2)--(.3,.2)--cycle;
\draw[fill=blueviolet](.25,0)--(.3,0)--(.3,.2)--(.25,.2)--cycle;

\draw[fill=bananamania] (.13,0) -- (.15,0) -- (.15,.2) -- (.13,.2)--cycle;
\draw[fill=bluegray] (.11,0) -- (.13,0) -- (.13,.5) -- (.11,.5)--cycle;
\draw[fill=capri] (.09,0) -- (.11,0) -- (.11,.5) -- (.09,.5)--cycle;
\draw[fill=brandeisblue] (.08,0) -- (.09,0) -- (.09,.5) -- (.08,.5)--cycle;

\draw[draw=black] (.08,0)--(.9,0)--(.9,-.2)--(1,-.2)--(1,.2)--(.13,.2)--(.13,.5)--(.08,.5)--(.08,0);
\node[below] at(.08,0){\tiny $\alpha$};
\node[below]at(.88,-0.2){\tiny $1-\alpha$};
\node[below]at(1,-.2){\tiny $1$};
\node[right]at(1,0){\tiny $0$};
\node[right]at(1,-.2){\tiny $-1$};
\node[right]at(1,.2){\tiny $1$};
\node[left]at(.08,.5){\tiny $\infty$};
\node[below]at(.5,0){\small $\frac{1}{2}$};
\node[below]at(.33,0){\small $\frac{1}{3}$};
\draw[dotted](.09,0)--(.09,-.04)node[below]{\small $\frac{1}{n}$};
\node[above]at(.11,.5){\small $\frac{1}{n-1+\alpha}$};
\node[above]at(.22,.2){\small $\frac{1}{n-1}$};
\node[above]at(.94,.2){\small $\frac{1}{1+\alpha}$};
\node[above]at(.47,.2){\small $\frac{1}{2+\alpha}$};
\node[right]at(1,.03){\tiny $\textcolor{white}{1/(n-1)}$};
\node[right]at(1,-.019){\tiny $\textcolor{white}{-1/(n+1)}$};
\draw[dotted](.13,0)--(.13,-.15)node[below]{\small $\frac{\alpha}{1-\alpha}$};
\filldraw[draw=black, fill=black] (.08,0) circle (.1pt);

\filldraw[draw=black, fill=black] (.09,0) circle (.1pt);
\filldraw[draw=black, fill=black] (.15,.2) circle (.1pt);
\filldraw[draw=black, fill=black] (.33,0) circle (.1pt);
\filldraw[draw=black, fill=black] (.5,0) circle (.1pt);
\filldraw[draw=black, fill=black] (.9,-.2) circle (.1pt);
\filldraw[draw=black, fill=black] (1,0) circle (.1pt);
\filldraw[draw=black, fill=black] (1,-.2) circle (.1pt);
\filldraw[draw=black, fill=black] (1,.2) circle (.1pt);
\filldraw[draw=black, fill=black] (.13,0) circle (.1pt);
\filldraw[draw=black, fill=black] (.94,.2) circle (.1pt);
\filldraw[draw=black, fill=black] (.47,.2) circle (.1pt);
\filldraw[draw=black, fill=black] (.11,.5) circle (.1pt);
\node at(.2,.1){\tiny $\ldots$};
\end{tikzpicture}}
\subfigure{
\begin{tikzpicture}[scale=5.2]
\draw[fill=brightpink] (.94,-.2) -- (1,-.2) -- (1,.2) -- (.94,.2)--cycle;
\draw[fill=aquamarine] (.9,-.2) -- (.94,-.2) -- (.94,.2) --(.5,.2)--(.5,0)-- (.9,0)--cycle;

\draw[fill=brightube] (.47,0)--(.5,0)--(.5,.2)--(.47,.2)--cycle;
\draw[fill=brilliantlavender] (.33,0) -- (.47,0) -- (.47,.2) -- (.33,.2)--cycle;
\draw[fill=fandango] (.3,0)--(.33,0)--(.33,.2)--(.3,.2)--cycle;
\draw[fill=blueviolet] (.25,0)--(.3,0)--(.3,.2)--(.25,.2)--cycle;

\draw[fill=bananamania] (.115,0) -- (.13,0) -- (.13,.2) -- (.115,.2)--cycle;
\draw[fill=bluegray] (.1,0) -- (.115,0) -- (.115,.5) -- (.1,.5)--cycle;
\draw[fill=capri] (.09,0) -- (.1,0) -- (.1,.5) -- (.09,.5)--cycle;
\draw[fill=brandeisblue] (.08,0) -- (.09,0) -- (.09,.5) -- (.08,.5)--cycle;

\draw [draw=black] (.08,0)--(.9,0)--(.9,-.2)--(1,-.2)--(1,.2)--(.115,.2)--(.115,.5)--(.08,.5)--(.08,0);
\node[below] at(.08,0){\tiny $\alpha$};
\draw[dotted](.115,0)--(.115,-.15)node[below]{\small $\frac{\alpha}{1-\alpha}$};
\node[below]at(.88,-.2){\tiny $1-\alpha$};
\node[right]at(1,0){\tiny $0$};
\node[right]at(1,-.2){\tiny $-1$};
\node[below]at(1,-.2){\tiny $1$};
\node[right]at(1,.2){\tiny $1$};
\node[left]at(.08,.5){\tiny $\infty$};
\node[below]at(.5,0){\small $\frac{1}{2}$};
\node[below]at(.33,0){\small $\frac{1}{3}$};
\node[above]at(.09,.5){\small $\frac{1}{n+\alpha}$};
\draw[dotted](.1,0)--(.1,-.04)node[below]{\small $\frac{1}{n}$};
\node[above]at(.24,.2){\small $\frac{1}{n-1+\alpha}$};
\node[above]at(.94,.2){\small $\frac{1}{1+\alpha}$};
\node[above]at(.47,.2){\small $\frac{1}{2+\alpha}$};
\node[right]at(1,.03){\tiny $\textcolor{white}{1/(n-1)}$};
\node[right]at(1,-.019){\tiny $\textcolor{white}{-1/(n+1)}$};
\draw[dotted](.09,0)--(.09,-.12)(.1,0)--(.1,-.06);
\filldraw[draw=black, fill=black] (.08,0) circle (.1pt);
\filldraw[draw=black, fill=black] (.09,.5) circle (.1pt);
\filldraw[draw=black, fill=black] (.1,0) circle (.1pt);
\filldraw[draw=black, fill=black] (.33,0) circle (.1pt);
\filldraw[draw=black, fill=black] (.5,0) circle (.1pt);
\filldraw[draw=black, fill=black] (.9,-.2) circle (.1pt);
\filldraw[draw=black, fill=black] (1,0) circle (.1pt);
\filldraw[draw=black, fill=black] (1,-.2) circle (.1pt);
\filldraw[draw=black, fill=black] (1,.2) circle (.1pt);
\filldraw[draw=black, fill=black] (.115,0) circle (.1pt);
\filldraw[draw=black, fill=black] (.94,.2) circle (.1pt);
\filldraw[draw=black, fill=black] (.47,.2) circle (.1pt);
\filldraw[draw=black, fill=black] (.13,.2) circle (.1pt);
\draw[dotted](.13,0)--(.13,.2);
\draw(.09,0)--(.09,.5)(.1,0)--(.1,.5)(.115,0)--(.115,.5)(.13,0)--(.13,.2);
\draw(.5,0)--(.5,.2)(.33,0)--(.33,.2);
\node at(.19,.1){\tiny$\ldots$};
\end{tikzpicture}}
\setcounter{subfigure}{0}
\subfigure[$\mathcal{D}_{\alpha}$ (above) and $\mathcal{T}_{\alpha}(\mathcal{D}_{\alpha})$ (below) \newline for $\frac{1}{n+\alpha} < \alpha < \frac{1}{n}$, $n \geq 2$]{
\begin{tikzpicture}[scale=5.2]
\draw[fill=brightpink] (.9,-.2) -- (1,-.2) -- (1,-.066) -- (.9,-.066)--cycle;
\draw[fill=brightube] (.9,-.066) -- (1,-.066) -- (1,-.05) -- (.9,-.05) -- cycle;
\draw[fill=fandango](.9,-.05)--(1,-.05)--(1,-.04)--(.9,-.04)--cycle;

\draw[fill=bananamania] (.96,-.023) -- (1,-.023) -- (1,-.014) -- (.96,-.014)--cycle;
\draw[fill=brandeisblue] (.96,-.014) -- (1,-.014) -- (1,0) -- (.96,0)--cycle;
\draw[fill=bluegray] (.9,-.023) -- (.96,-.023) -- (.96,0) -- (.9,0)--cycle;

\draw[fill=capri] (.08,0) -- (1,0) -- (1,.02) -- (.08,.02)--cycle;

\draw[fill=brilliantlavender] (.08,.067)-- (1,.067)--(1,.1) -- (.08,.1) -- cycle;
\draw[fill=blueviolet](.08,.05)--(1,.05)--(1,.067)--(.08,.067)--cycle;
\draw[fill=aquamarine] (.08,.1) -- (1,.1) --(1,.2) -- (.13,.2) -- (.13,.5) -- (.08,.5)--cycle;

\draw [draw=black] (.08,0)--(.9,0)--(.9,-.2)--(1,-.2)--(1,.2)--(.13,.2)--(.13,.5)--(.08,.5)--(.08,0);
\node[below] at(.08,0){\tiny $\alpha$};
\node[above]at(.13,.5){\small $\frac{\alpha}{1-\alpha}$};
\node[below]at(.88,-.2){\tiny $1-\alpha$};
\node[left]at(.08,0){\tiny $0$ };
\node[right]at(1,-.2){\tiny $-1$};
\node[right]at(1,.2){\tiny $1$};
\node[below]at(1,-.2){\tiny $1$};
\node[left]at(.08,.5){\tiny $\infty$};
\node[left]at(.08,.11){\tiny $1/2$};
\node[left]at(.08,.066){\tiny $1/3$};
\node[right]at(1,.03){\tiny $1/(n-1)$};
\node[right]at(1,-.019){\tiny $-1/(n+1)$};
\node[left]at(.91,-.025){\tiny $-1/n$};
\draw[dotted](.96,-.014)--(.96,.21)node[above]{\small $\frac{\alpha-1+n\alpha}{\alpha}$};
\node[left]at(.9,-.08){\tiny $-1/3$};
\filldraw[draw=black, fill=black] (.08,0) circle (.1pt);
\filldraw[draw=black, fill=black] (.13,.5) circle (.1pt);
\filldraw[draw=black, fill=black] (.08,.1) circle (.1pt);
\filldraw[draw=black, fill=black] (.08,.067) circle (.1pt);
\filldraw[draw=black, fill=black] (1,-.2) circle (.1pt);
\filldraw[draw=black, fill=black] (1,.02) circle (.1pt);
\filldraw[draw=black, fill=black] (1,-.018) circle (.1pt);
\filldraw[draw=black, fill=black] (.9,-.023) circle (.1pt);
\filldraw[draw=black, fill=black] (.9,-.2) circle (.1pt);
\filldraw[draw=black, fill=black] (.96,.2) circle (.1pt);
\filldraw[draw=black, fill=black] (.9,-.066) circle (.1pt);

\node at(.5,.035){\tiny$\ldots$};
\node at(.95,-.03){\tiny $\ldots$};
\end{tikzpicture}}
\subfigure[$\mathcal{D}_{\alpha}$ (above) and $\mathcal{T}_{\alpha}(\mathcal{D}_{\alpha})$ (below) \newline for $\frac{1}{n+1} < \alpha < \frac{1}{n+ \alpha}$, $n \ge 2$]{
\begin{tikzpicture}[scale=5.2]
\draw[fill=brightpink] (.9,-.2) -- (1,-.2) -- (1,-.066) -- (.9,-.066)--cycle;
\draw[fill=brightube] (.9,-.066)--(1,-.066)--(1,-.05)--(.9,-.05)--cycle;
\draw[fill=fandango] (.9,-.05)--(1,-.05)--(1,-.04)--(.9,-.04)--cycle;

\draw[fill=aquamarine] (.08,.1) -- (1,.1)--(1,.2) -- (.115,.2) -- (.115,.5) -- (.08,.5)--cycle;
\draw[fill=brilliantlavender] (.08,.066)--(1,.066)--(1,.1)--(.08,.1)--cycle;
\draw[fill=blueviolet] (.08,.05)--(1,.05)--(1,.066)--(.08,.066)--cycle;

\draw[fill=bananamania] (.08,.016) -- (.45,.016) -- (.45,.025) -- (.08,.025)--cycle;
\draw[fill=brandeisblue] (.08,0) -- (.45,0) -- (.45,.016) -- (.08,.016)--cycle;
\draw[fill=bluegray] (.45,0)--(1,0)--(1,.025)--(.45,.025)--cycle;
\draw[fill=capri] (.9,0) -- (1,0) -- (1,-.02) -- (.9,-.02)--cycle;

\draw [draw=black] (.08,0)--(.9,0)--(.9,-.2)--(1,-.2)--(1,.2)--(.115,.2)--(.115,.5)--(.08,.5)--(.08,0);
\node[below] at(.08,0){\tiny $\alpha$};
\node[above]at(.115,.5){\small $\frac{\alpha}{1-\alpha}$};
\node[below]at(.88,-.2){\tiny $1-\alpha$};
\node[left]at(.08,0){\tiny $0$};
\node[right]at(1,-.2){\tiny $-1$};
\node[below]at(1,-.2){\tiny $1$};
\node[right]at(1,.2){\tiny $1$};
\node[left]at(.08,.5){\tiny $\infty$};
\node[left]at(.08,.1){\tiny $1/2$};
\node[left]at(.08,.066){\tiny $1/3$};
\node[right]at(1,.0295){\tiny $1/(n-1)$};
\node[right]at(1,-.021){\tiny $-1/(n+1)$};
\node[below]at(.45,0){\small $\frac{1-n\alpha}{\alpha}$};
\node[left]at(.9,-.066){\tiny $-1/3$};
\filldraw[draw=black, fill=black] (.08,0) circle (.1pt);
\filldraw[draw=black, fill=black] (.115,.5) circle (.1pt);
\filldraw[draw=black, fill=black] (.08,.1) circle (.1pt);
\filldraw[draw=black, fill=black] (.08,.066) circle (.1pt);
\filldraw[draw=black, fill=black] (1,-.2) circle (.1pt);
\filldraw[draw=black, fill=black] (1,.2) circle (.1pt);
\filldraw[draw=black, fill=black] (1,.025) circle (.1pt);
\filldraw[draw=black, fill=black] (1,.02) circle (.1pt);
\filldraw[draw=black, fill=black] (1,-.018) circle (.1pt);
\filldraw[draw=black, fill=black] (.9,-.2) circle (.1pt);
\filldraw[draw=black, fill=black] (.45,0) circle (.1pt);
\filldraw[draw=black, fill=black] (.9,-.066) circle (.1pt);
\node at(.5,.04){\tiny $\ldots$};
\node at(.95,-.03){\tiny$\ldots$};
\end{tikzpicture}}
\caption{The transformation $\mathcal{T}_{\alpha}$ maps areas on the top to areas on the bottom with the same colours, respectively.}
\label{f:ne1}
\end{center}
\end{figure}

\vskip .2cm
The properties (ne2) and (ne3) follow immediately. Left to prove are (ne4) and the fact that $\nu_\alpha$ is $\mathcal T_\alpha$-invariant. To prove that $\nu_\alpha$ is invariant for $\mathcal{T}_\alpha$, it suffices to check that $\nu_\alpha(A) = \nu_\alpha(\mathcal T^{-1}_\alpha(A))$ for any rectangle $A=[a,b]\times [c,d]\subseteq \mathcal T_\alpha (D)$ for any $D = \tilde \Delta(\epsilon,d)$, $D = \tilde \Delta_L(\epsilon,d)$ or $D = \tilde \Delta_R(\epsilon,d)$. This computation is very similar to the corresponding ones for natural extensions of other continued fraction maps that can be found in literature. We refer to \cite[Theorem 1]{Nak81} for an example of such a computation.

\vskip .2cm
To prove that (ne4) holds, it is enough to show that $\bigvee_{n=0}^\infty \mathcal T_\alpha^n \pi^{-1}(\mathcal B_\alpha)$ separates points, i.e., that for $\lambda^2$-almost all $(x,y),(x',y') \in \mathcal D_\alpha$ with $(x,y)\neq (x',y')$ there are disjoint sets $A,B \in \bigvee_{n=0}^\infty \mathcal T_\alpha^n \pi^{-1}(\mathcal B_\alpha)$ with $(x,y)\in A$ and $(x',y') \in B$. Since $\mathcal B_\alpha$ is separating, the property is clear if $x \neq x'$. Furthermore, note that for $\lambda$-almost all values of $y$ there is an $\varepsilon$ and a $ d$, such that on a neighbourhood of $(x,y)$, the inverse of $\mathcal T_\alpha$ is given by
\[
\mathcal T_\alpha^{-1}(x,y) = \left(\frac{1}{d+\varepsilon x} , \frac{\varepsilon}{y}-d\right).
\]
The map $\frac{\varepsilon}{y}-d$ is expanding and $\mathcal T_\alpha^{-1}$ maps horizontal strips to vertical strips. Hence, we can also separate points that agree on the $x$-coordinate, giving (ne4). Therefore, we have obtained the following result.
\begin{thm}
Let $\alpha \in \big(0, \frac12\big)$. The dynamical system $(\mathcal D_\alpha, \mathcal B(\mathcal D_\alpha), \nu_\alpha, \mathcal T_\alpha)$ is a version of the natural extension of the dynamical system $(I_\alpha, \mathcal B_\alpha, \mu_\alpha, T_\alpha)$ where $\mu_\alpha :=\nu_\alpha \circ \pi^{-1}$.
\end{thm}

The measure $\mu_\alpha = \nu_\alpha \circ \pi^{-1}$ is the unique invariant measure for $T_\alpha$ that is absolutely continuous with respect to $\lambda$ from Proposition~\ref{t:acim}. Projecting on the first coordinate gives the following explicit expression for the density $f_\alpha$ of $\mu_\alpha$:
\begin{equation}\label{q:density} \begin{split}
f_\alpha (x) = \ & \frac{1}{x} \mathbf{1}_{[\alpha,\frac{\alpha}{1-\alpha}]} (x) + \frac{1}{1+x} \mathbf{1}_{[\frac{\alpha}{1-\alpha},1]} (x)+ \frac{1}{1-x} \mathbf{1}_{[1-\alpha,1]}(x)\\
=\ & \frac{1}{x}  \mathbf{1}_{[\alpha,\frac{\alpha}{1-\alpha}]} (x)+ \frac{1}{1+x} \mathbf{1}_{[\frac{\alpha}{1-\alpha},1-\alpha]}(x) + \frac{2}{1-x^2} \mathbf{1}_{[1-\alpha,1]}(x).
\end{split}
\end{equation}
Here, by ``unique'', we of course mean unique up to scalar multiples. We choose to work with the above expression, because it comes from projecting the canonical measure \eqref{e:canme} for the natural extension, and is thus a natural choice.

\subsection{Natural extension $\mathbf{T_{\alpha}}$ for $ \mathbf{\alpha \geq \frac12}$}
As indicated by Theorem~\ref{t:acf} the situation for $\alpha \geq \frac12$ becomes increasingly complicated. Figure~\ref{f:nevar} shows the natural extension domain $\mathcal D_\alpha$ for $\alpha \in \big[ \frac12, \frac12 \sqrt 2 \big)$ with the action of $\mathcal T_\alpha$ and Table~\ref{t:densities} provides the corresponding densities. We do not provide further details as the proofs are exactly like the one for $0 < \alpha < \frac12$.

\begin{figure}[h]
\centering
\subfigure{
\begin{tikzpicture}[scale=3.5]
\draw[draw=black,fill=babyblueeyes] (.08,-.2)--(1,-.2)--(1,0)--(.85,0)--(.85,.1)--(1,.1)--(1,.3)--(.27,.3)--(.27,0)--(.08,0)--cycle;
\draw[draw=black,fill=aquamarine] (.08,-.2) -- (.15,-.2) -- (.15,0) -- (.08,0)--cycle;
\draw[draw=black,fill=blush] (.15,-.2) -- (.6,-.2) -- (.6,.3) -- (.27,.3)--(.27,0)--(.15,0)--cycle;

\draw [draw=black] (.08,-.2)--(1,-.2)--(1,0)--(.85,0)--(.85,.1)--(1,.1)--(1,.3)--(.27,.3)--(.27,0)--(.08,0)--cycle;
\draw [draw=black, dashed] (.15,-.2)--(.15,0)(.27,-.2)--(.27,0)(.6,-.2)--(.6,.3)(.85,-.2)--(.85,.3);
\draw[draw=black, dashed] (1,0)--(.08,0);

\node[below] at(.27,-.2){\tiny $\alpha$};
\node[below]at(.02,-.2){\tiny $1-\alpha$};
\node[below]at(.6,-.2){\tiny $\frac{1}{1+\alpha}$};
\node[below]at(.15,-.2){\tiny $\frac{1}{2}$};
\node[below]at(.85,-.2){\tiny $\frac{1-\alpha}{\alpha }$};

\node[right]at(1,0){\tiny $0$};
\node[right]at(1,-.2){\tiny $-1$};
\node[right]at(1,.1){\tiny $1$};
\node[right]at(1,.3){\tiny $\infty$};
\end{tikzpicture}}
\subfigure{
\begin{tikzpicture}[scale=3.5]
\draw[draw=black,fill=bananamania] (.08,-.2)--(.15,-.2)--(.15,-.05)--(.08,-.05)--cycle;
\draw[draw=black,fill=babyblueeyes] (.15,-.2)--(.4,-.2)--(.4,0)--(.35,0)--(.35,-.05)--(.15,-.05)--cycle;
\draw[draw=black,fill=brilliantlavender] (.4,-.2)--(.45,-.2)--(.45,0)--(.4,0)--cycle;
\draw[draw=black,fill=aquamarine] (.45,-.2)--(1,-.2)--(1,0)--(.45,0)--cycle;
\draw[draw=black,fill=aquamarine] (.8,.05) -- (.95,.05)--(.95,.15)--(1,.15)--(1,.4)--(.8,.4)--(.8,.2)--cycle;

\draw[draw=black] (.15,-.2)--(.15,-.05)(.45,-.2)--(.45,0)(.4,-.2)--(.4,0);
\draw[draw=black] (.08,-.2)--(1,-.2)--(1,0)--(.35,0)--(.35,-.05)--(.08,-.05)--(.08,-.2);
\draw[draw=black] (.8,.05) -- (.95,.05)--(.95,.15)--(1,.15)--(1,.4)--(.8,.4)--(.8,.2)--cycle;
\draw[draw=black, dashed] (.95,-.2)--(.95,.4);
\draw[draw=black, dashed] (.08,-.05)--(1,-.05);
\draw[draw=black, dashed] (.35,-.2)--(.35,-.05);

\node[below] at(.8,.075){\tiny $\alpha$};
\node[below]at(.01,-.18){\tiny $1-\alpha$};
\node[below]at(.45,-.21){\tiny $\frac{1}{1+\alpha}$};
\node[below]at(.95,-.2){\tiny $\frac{2\alpha -1}{1-\alpha}$};
\node[below]at(.32,-.18){\tiny $\frac{2\alpha-1}{\alpha }$};
\node[below]at(.15,-.2){\tiny $\frac{1}{2+\alpha}$};
\node[above]at(.4,0){\tiny $\frac{1}{2}$};

\node[right]at(1,0){\tiny $0$};
\node[right]at(1,-.2){\tiny $-1$};
\node[right]at(1,-.05){\tiny $g-1$};
\node[right]at(1,.15){\tiny $1$};
\node[right]at(1,.4){\tiny $\infty$};

\node[left]at(.8,.05){\tiny $g$};
\end{tikzpicture}}
\subfigure{
\begin{tikzpicture}[scale=3.5]
\draw[draw=black,fill=bananamania] (.08,-.2) -- (.15,-.2) -- (.15,-.05) -- (.08,-.05)--cycle;
\draw[draw=black,fill=blush] (.15,-.2) -- (.17,-.2) -- (.17,-.05) -- (.15,-.05)--cycle;
\draw[draw=black,fill=babyblueeyes] (.17,-.2) -- (.25,-.2) -- (.25,-.05) -- (.17,-.05)--cycle;
\draw[draw=black,fill=brilliantlavender] (.25,-.2) -- (.4,-.2) -- (.4,.0) -- (.35,.0)--(.35,-.05)--(.25,-.05)--cycle;
\draw[draw=black,fill=aquamarine] (.4,-.2) -- (1,-.2) -- (1,0) -- (.4,0)--cycle;
\draw[draw=black,fill=aquamarine] (.8,.05) -- (1,.05) -- (1,.1) -- (.95,.1)--(.95,.15)--(1,.15)--(1,.4)--(.8,.4)--(.8,.2)--cycle;

\draw[draw=black] (.08,-.2)--(1,-.2)--(1,0)--(.35,0)--(.35,-.05)--(.08,-.05)--(.08,-.2);
\draw[draw=black] (.8,.05) -- (1,.05) -- (1,.1) -- (.95,.1)--(.95,.15)--(1,.15)--(1,.4)--(.8,.4)--(.8,.2)--cycle;
\draw[draw=black] (.15,-.2)--(.15,-.05) (.17,-.2)--(.17,-.05) (.25,-.2)--(.25,-.05) (.4,-.2)--(.4,.0);
\draw[draw=black, dashed] (.95,-.2)--(.95,.4);
\draw[draw=black, dashed] (.08,-.05)--(1,-.05);
\draw[draw=black, dashed] (.35,-.2)--(.35,-.05);

\node[below] at(.8,.075){\tiny $\alpha$};
\node[below]at(.08,-.2){\tiny $1-\alpha$};
\node[below]at(.41,-.2){\tiny $\frac{1}{1+\alpha}$};
\node[above]at(.25,-.06){\tiny $\frac{1}{2}$};
\node[above]at(.15,-.06){\tiny $\frac{1}{3}$};
\node[below]at(.95,-.2){\tiny $\frac{1-\alpha}{2\alpha -1}$};
\node[above]at(.36,-.01){\tiny $\frac{2\alpha-1}{\alpha }$};

\node[right]at(1,0){\tiny $0$};
\node[right]at(1,-.2){\tiny $-1$};
\node[right]at(1,-.05){\tiny $g-1$};
\node[right]at(1,.05){\tiny $g$};
\node[right]at(1,.1){\tiny $1$};
\node[right]at(1,.15){\tiny $g+1$};
\node[right]at(1,.4){\tiny $\infty$};

\end{tikzpicture}}

\setcounter{subfigure}{0}
\subfigure[][$\alpha \in [\frac12,g)$]{
\begin{tikzpicture}[scale=3.5]
\draw[draw=black,fill=babyblueeyes] (.08,-.2)--(1,-.2)--(1,0)--(.85,0)--(.85,.1)--(1,.1)--(1,.3)--(.27,.3)--(.27,0)--(.08,0)--cycle;
\draw[draw=black,fill=aquamarine] (.45,-.1) -- (1,-.1) -- (1,-.05) -- (.45,-.05)--cycle;
\draw[draw=black,fill=blush] (.27,0) -- (.85,0) -- (.85,.1) -- (1,.1)--(1,.3)--(.27,.3)--cycle;

\draw[draw=black] (.08,-.2)--(1,-.2)--(1,0)--(.85,0)--(.85,.1)--(1,.1)--(1,.3)--(.27,.3)--(.27,0)--(.08,0)--cycle;
\draw[draw=black, dashed] (.45,-.2)--(.45,-.05)(.27,-.2)--(.27,0);
\draw[draw=black, dashed] (.08,-.1)--(.45,-.1)(.08,-.05)--(.45,-.05);

\node[below] at(.27,-.2){\tiny $\alpha$};
\node[below]at(0.08,-.2){\tiny $1-\alpha$};
\node[below]at(.45,-.2){\tiny $\frac{3\alpha-2}{\alpha -1}$};

\node[right]at(1,0){\tiny $0$};
\node[right]at(1,-.2){\tiny $-1$};
\node[right]at(1.03,-.1){\tiny $-\frac12$};
\node[right]at(1,-.05){\tiny $-\frac13$};
\node[right]at(1,.1){\tiny $1$};
\node[right]at(1,.3){\tiny $\infty$};
\end{tikzpicture}}
\subfigure[$\alpha \in [g, \frac23)$]{

\begin{tikzpicture}[scale=3.5]
\draw[draw=black,fill=bananamania] (.8,.05) -- (.95,.05)--(.95,.15)--(.8,.15)--cycle;
\draw[draw=black,fill=babyblueeyes] (.08,-.1)--(1,-.1)--(1,-.02)--(.5,-.02)--(.5,-.05)--(.08,-.05)--cycle;
\draw[draw=black,fill=brilliantlavender] (.8,.15)--(1,.15)--(1,.4)--(.8,.4)--cycle;
\draw[draw=black,fill=aquamarine] (.08,-.2)--(1,-.2)--(1,-.1)--(.08,-.1)--cycle;
\draw[draw=black,fill=aquamarine] (.35,-.05)--(.5,-.05)--(.5,-.02)--(1,-.02)--(1,0)--(.5,0)--(.35,0)--cycle;

\draw [draw=black] (.08,-.2)--(1,-.2)--(1,0)--(.35,0)--(.35,-.05)--(.08,-.05)--(.08,-.2);
\draw[draw=black] (.8,.05) -- (.95,.05)--(.95,.15)--(1,.15)--(1,.4)--(.8,.4)--(.8,.2)--cycle;
\draw[draw=black, dashed] (.95,-.2)--(.95,.4);
\draw[draw=black, dashed] (.08,-.05)--(1,-.05);
\draw[draw=black, dashed] (.5,-.2)--(.5,0);
\draw[draw=black, dashed] (.08,-.1)--(1,-.1)(.35,-.02)--(.5,-.02);
\draw[draw=black, dashed] (.35,-.2)--(.35,-.05);

\node[below]at(.08,-.2){\tiny $1-\alpha$};
\node[below]at(.5,-.2){\tiny $\frac{5\alpha-3}{2\alpha-1}$};
\node[below]at(.95,-.2){\tiny $\frac{2\alpha -1}{1-\alpha}$};
\node[below]at(.32,-.2){\tiny $\frac{2\alpha-1}{\alpha }$};

\node[right]at(1,0){\tiny $0$};
\node[right]at(1,-.2){\tiny $-1$};
\node[right]at(1,-.05){\tiny $g-1$};
\node[right]at(1,-.1){\tiny $-\frac12$};
\node[right]at(1,.15){\tiny $1$};
\node[right]at(1,.4){\tiny $\infty$};

\node[left]at(.35,-.01){\tiny $-\frac13$};
\end{tikzpicture}}
\subfigure[$\alpha \in [\frac23,\frac{1}{2}\sqrt{2})$]{

\begin{tikzpicture}[scale=3.5]
\draw[draw=black,fill=aquamarine] (.08,-.2)--(1,-.2)--(1,0)--(.35,0)--(.35,-.05)--(.08,-.05)--(.08,-.2)--cycle;
\draw[draw=black,fill=bananamania] (.45,-.02) -- (1,-.02) -- (1,-.01) -- (.45,-.01)--cycle;
\draw[draw=black,fill=brilliantlavender] (.8,.05) -- (1,.05) -- (1,.1) -- (.95,.1)--(.95,.15)--(1,.15)--(1,.4)--(.8,.4)--(.8,.2)--cycle;
\draw[draw=black,fill=blush] (.8,.05) -- (1,.05) -- (1,.1) --(.8,.1)--cycle;
\draw[draw=black,fill=babyblueeyes] (.08,-.1) -- (1,-.1) -- (1,-.05) -- (.08,-.05)--cycle;

\draw[draw=black] (.08,-.2)--(1,-.2)--(1,0)--(.35,0)--(.35,-.05)--(.08,-.05)--(.08,-.2);
\draw[draw=black] (.8,.05) -- (1,.05) -- (1,.1) -- (.95,.1)--(.95,.15)--(1,.15)--(1,.4)--(.8,.4)--(.8,.2)--cycle;
\draw[draw=black, dashed] (.95,-.2)--(.95,.4);
\draw[draw=black, dashed] (.08,-.05)--(1,-.05);
\draw[draw=black, dashed] (.45,-.2)--(.45,0);
\draw[draw=black, dashed] (.08,-.1)--(1,-.1);
\draw[draw=black, dashed] (.35,-.2)--(.35,-.05);
\draw[draw=black, dashed] (.35,-.02)--(.45,-.02);
\draw[draw=black, dashed] (.35,-.01)--(.45,-.01);

\node[below] at(.8,.07){\tiny $\alpha$};
\node[below]at(.08,-.2){\tiny $1-\alpha$};
\node[below]at(.47,-.2){\tiny $\frac{3-4\alpha}{1-\alpha}$};
\node[below]at(.95,-.2){\tiny $\frac{1-\alpha}{2\alpha -1}$};
\node[below]at(.3,-.2){\tiny $\frac{2\alpha-1}{\alpha }$};

\node[right]at(1,-.2){\tiny $-1$};
\node[right]at(1,.05){\tiny $g$};
\node[right]at(1,-.1){\tiny $-\frac12$};
\node[right]at(1,.1){\tiny $1$};
\node[right]at(1,.15){\tiny $g+1$};
\node[right]at(1,.4){\tiny $\infty$};

\node[right]at(1.02,-.03){\tiny $-\frac13$};
\node[left]at(.33,0){\tiny $-\frac{1}{g+3}$};

\end{tikzpicture}}
\caption{The maps $\mathcal T_\alpha$ for various values of $\alpha$. Areas on the top are mapped to areas on the bottom with the same color.}\label{f:nevar}
\end{figure}
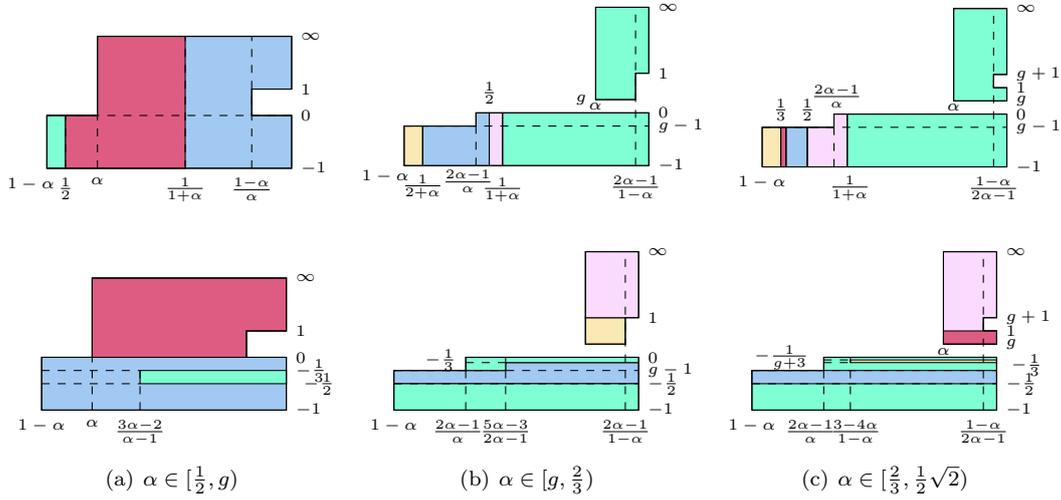

\bgroup
\def\arraystretch{1.75}
\begin{table}[h!]
  \begin{center}
    \begin{tabular}{l|l}
      \textbf{$\alpha$} & \textbf{Density $f_\alpha$}\\
      \hline
      $[\frac12, g)$ & $\frac{1}{1-x} \mathbf 1_{[1-\alpha, \alpha]}(x) + \frac1{x(1-x)} \mathbf 1_{[\alpha, \frac{1-\alpha}{\alpha}  ]} (x)+ \frac{x^2+1}{x(1-x^2)} \mathbf 1_{[\frac{1-\alpha}{\alpha},1]} (x)$ \\
      $[g, \frac23)$ & $\big(\frac{1}{1-x} + \frac{1}{x+\frac{1}{g-1}}\big)\mathbf 1_{[1-\alpha, \frac{2\alpha-1}{\alpha}]} (x) + \frac1{1-x} \mathbf 1_{[\frac{2\alpha-1}{\alpha}, \alpha ]} (x)+$  \\
   &   $+ \big(\frac{1}{1-x} + \frac{1}{x} -\frac{1}{x+\frac{1}{g}}\big) \mathbf 1_{[\alpha,\frac{2\alpha-1}{1-\alpha}]} (x)+  \frac{x^2+1}{x(1-x^2)} \mathbf 1_{[\frac{2\alpha-1}{1-\alpha},1]} (x) $ \\
      $[\frac23, \frac12 \sqrt{2})$ & $\big(\frac{1}{1-x} + \frac{1}{x+\frac{1}{g-1}}\big)\mathbf 1_{[1-\alpha, \frac{2\alpha-1}{\alpha}]} (x) + \frac1{1-x} \mathbf 1_{[\frac{2\alpha-1}{\alpha}, \alpha ]} (x) + \big(\frac{1}{1-x} + \frac{1}{x} -\frac{1}{x+\frac{1}{g}}\big) \mathbf 1_{[\alpha,\frac{1-\alpha}{2\alpha-1}]} (x)+$\\
    &  $ +\big(\frac{1}{1-x}+\frac{1}{x+1} -\frac{1}{x+ \frac1g} + \frac{1}{x} - \frac{1}{x+\frac{1}{g+1}}\big) \mathbf 1_{[\frac{1-\alpha}{2\alpha-1},1]} (x)$ \\
    \end{tabular}
    \vskip .2cm
     \caption{Invariant densities for $\alpha \in \big[\frac12, \frac12 \sqrt 2 \big)$.}
    \label{t:densities}
  \end{center}
\end{table}
\egroup

As $\alpha$ increases even further, the domain $\mathcal{D}_\alpha$ starts to exhibit a fractal structure. Figure~\ref{f:nesimulation} shows numerical simulations for various values of $\alpha > \frac 12\sqrt 2$.

\begin{figure}[h]
\centering
\begin{subfigure}[\text{$\alpha\approx 0.73694949$}]{
\includegraphics[scale=0.08]{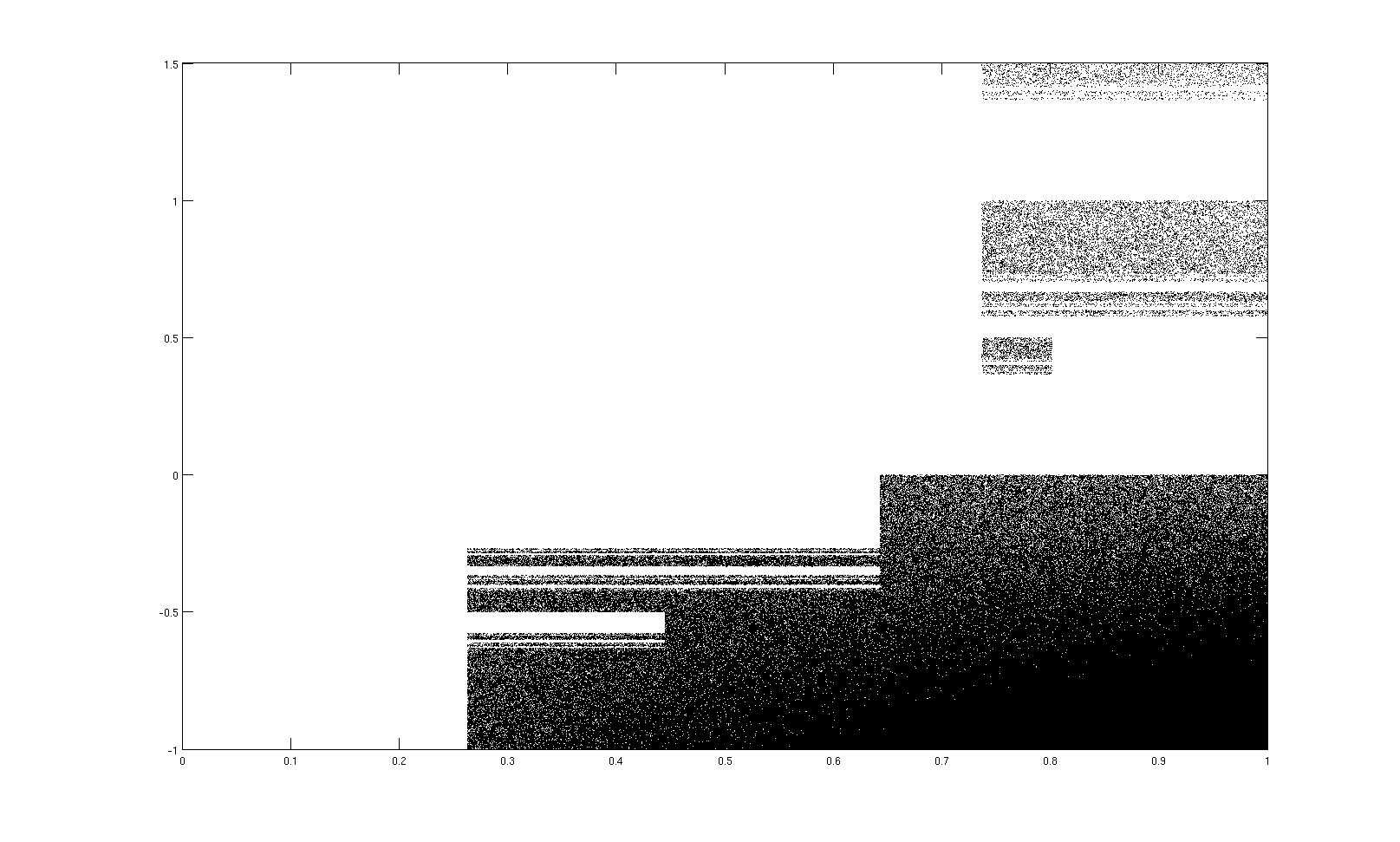}}
\end{subfigure}
\begin{subfigure}[\text{$\alpha\approx 0.79347519$}]{
\includegraphics[scale=0.08]{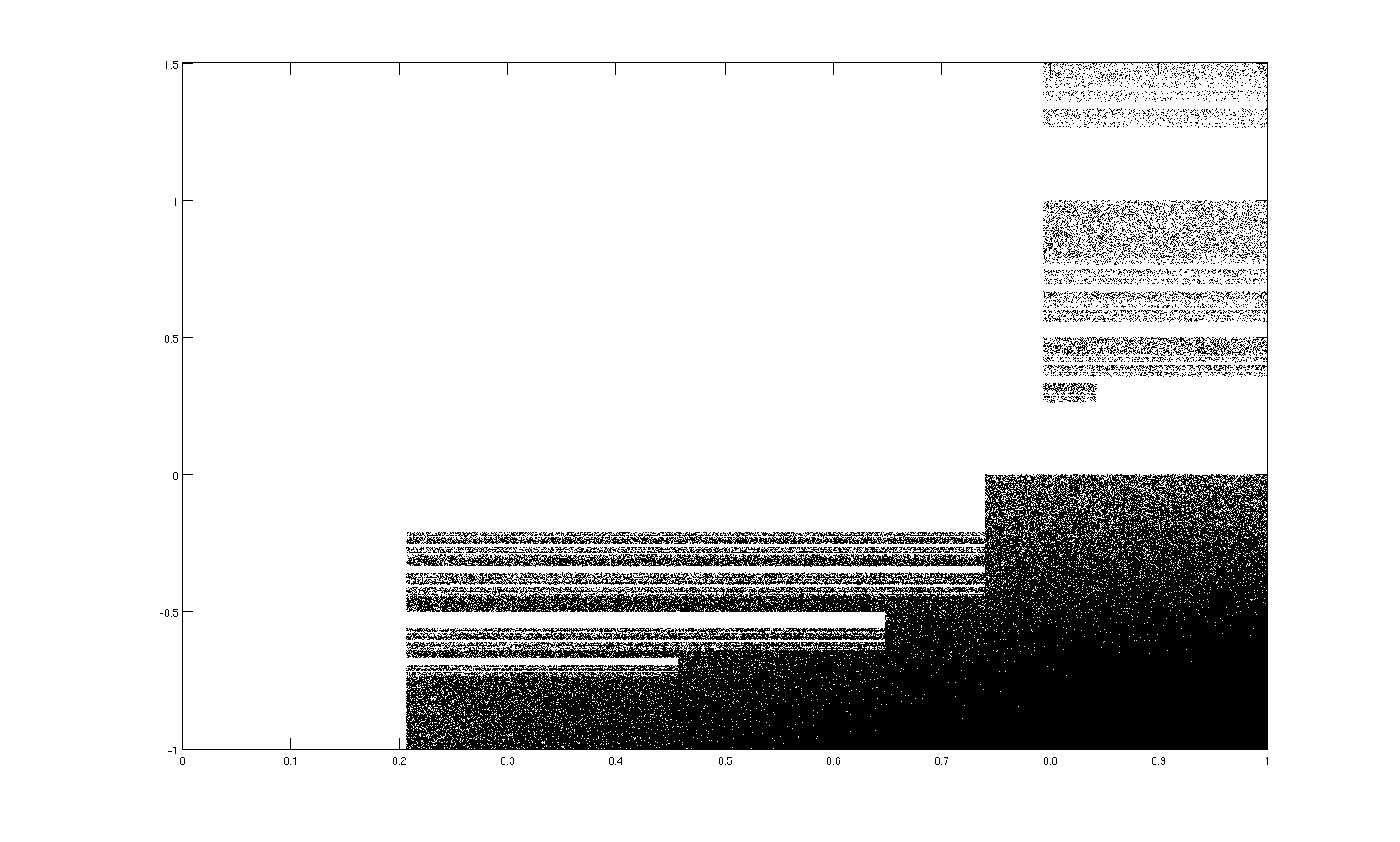}}
\end{subfigure}
\begin{subfigure}[\text{$\alpha \approx 0.85019348$}]{
\includegraphics[scale=0.08]{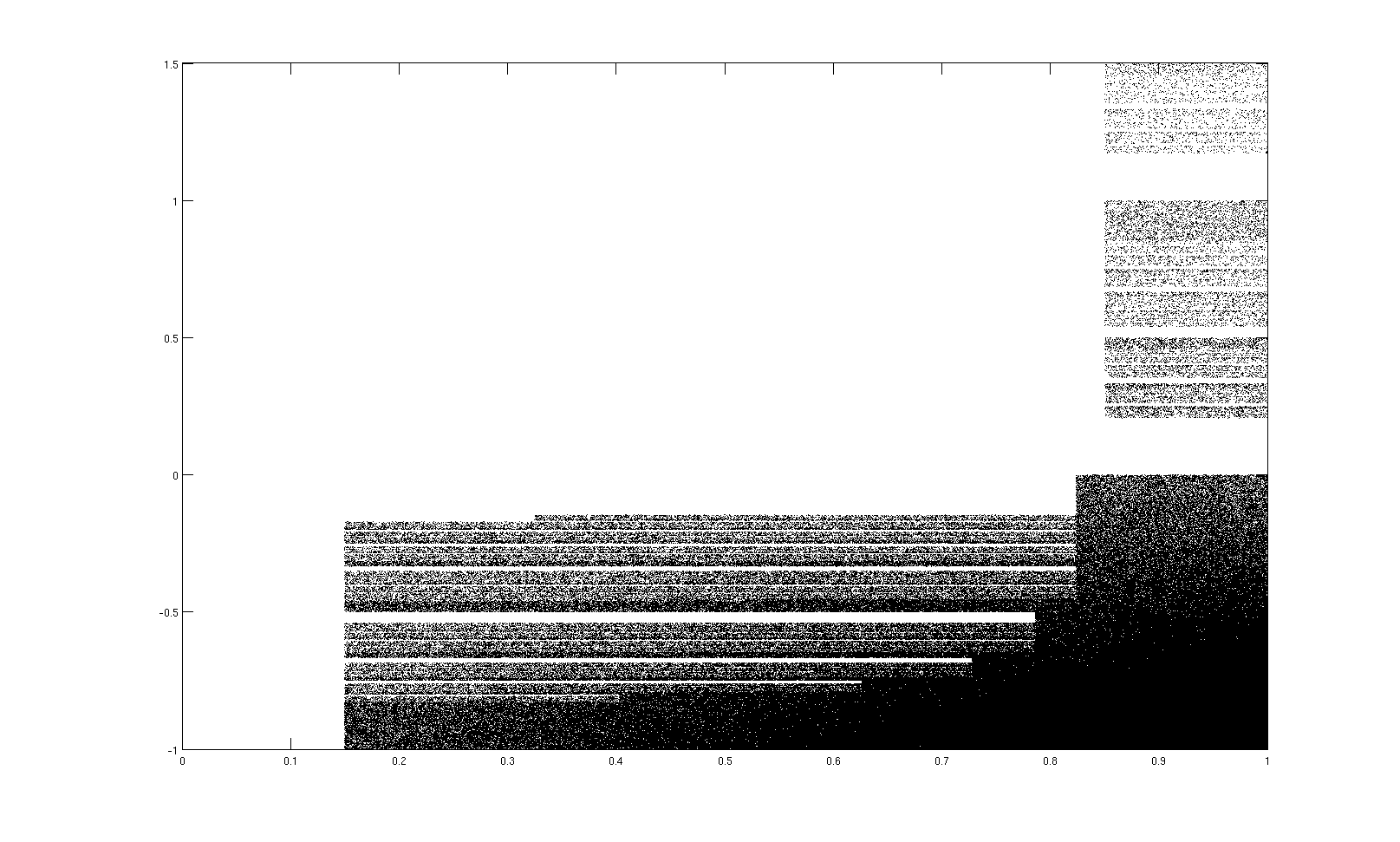}}
\end{subfigure}
\begin{subfigure}[\text{$\alpha \approx 0.89348572$}]{
\centering
\includegraphics[scale=0.08]{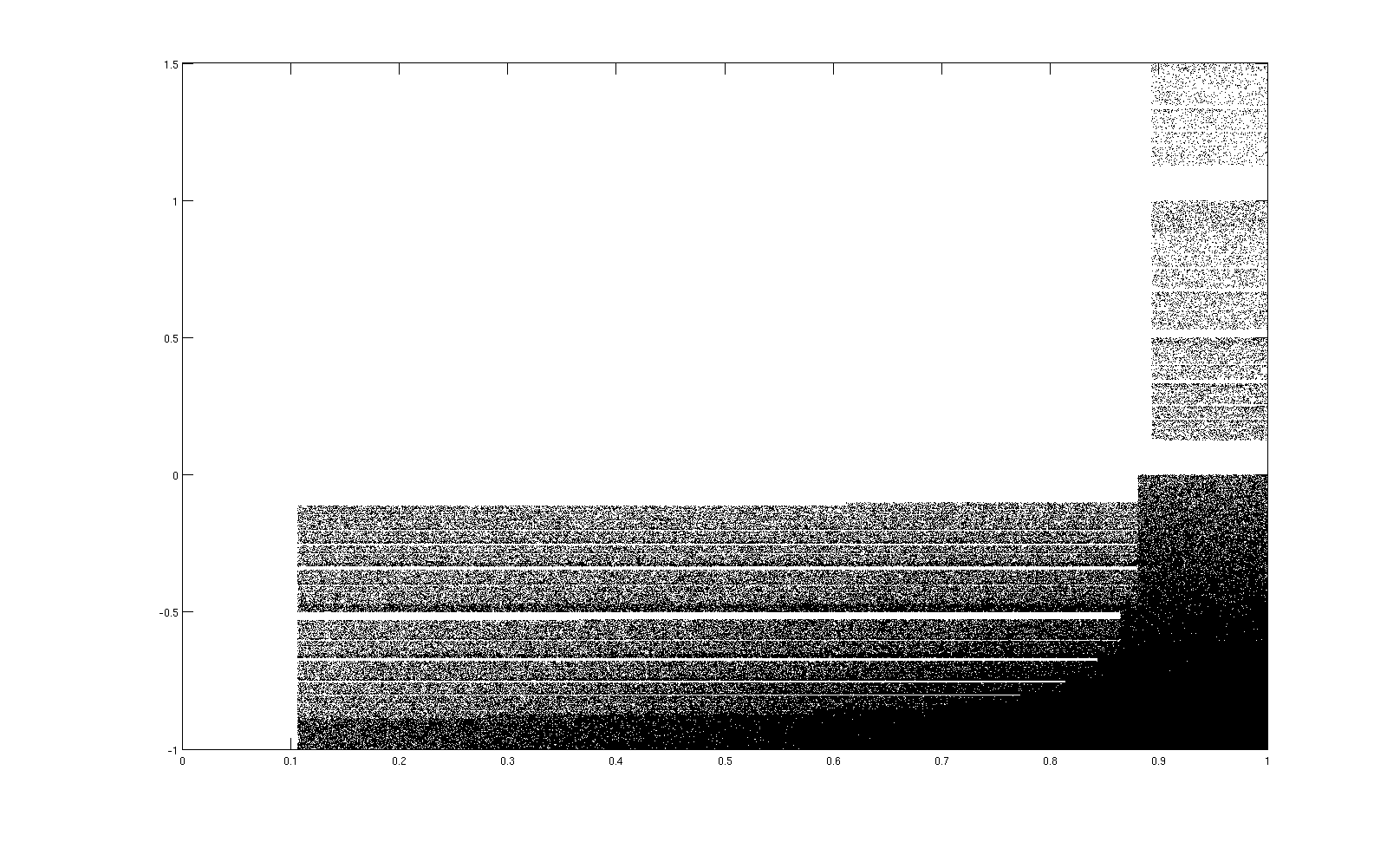}}
\end{subfigure}
\begin{subfigure}[\text{$\alpha \approx 0.92087668$}]{
\centering
\includegraphics[scale=0.08]{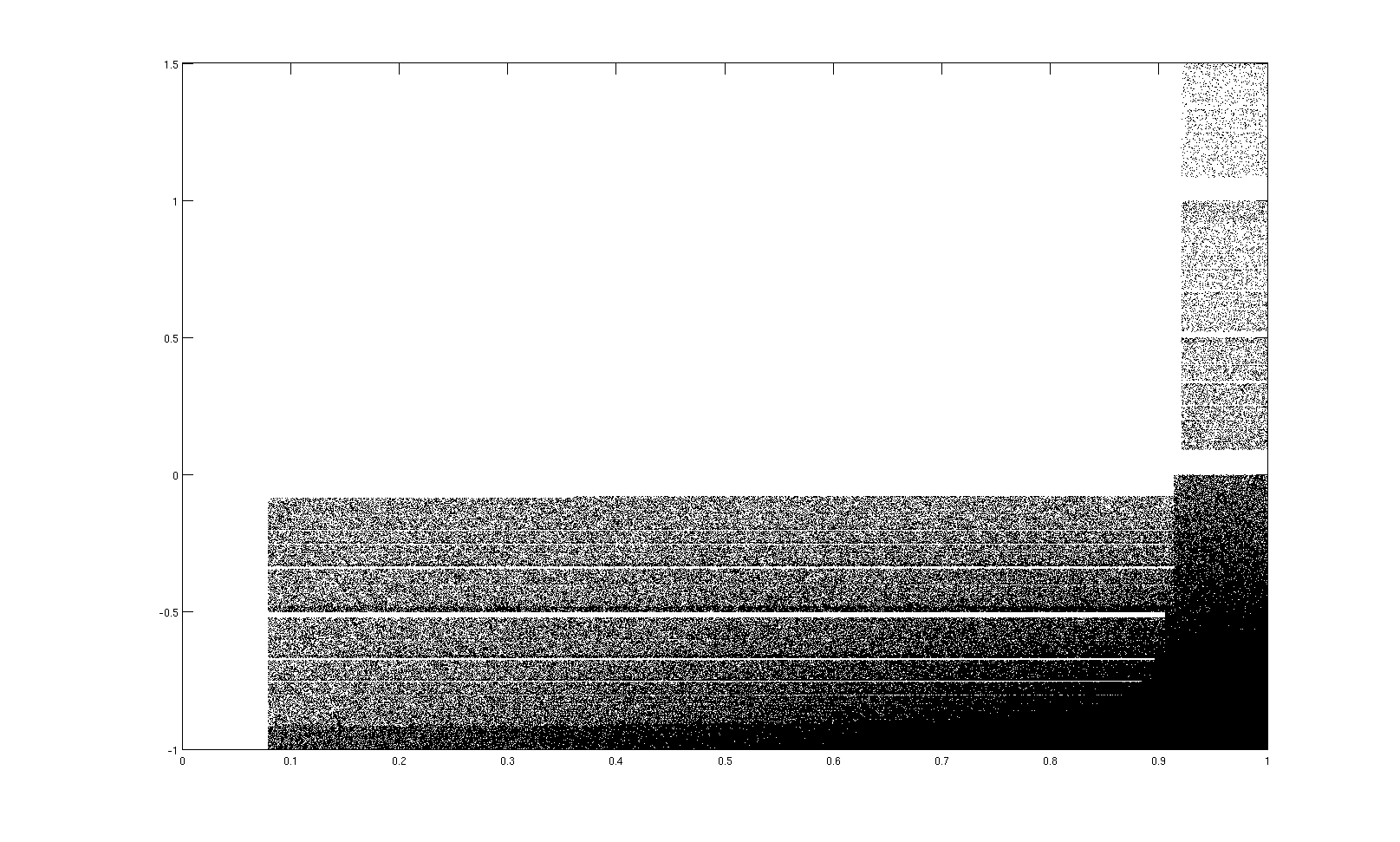}}
\end{subfigure}
\begin{subfigure}[\text{$\alpha \approx 0.95234649$}]{
\centering
\includegraphics[scale=0.08]{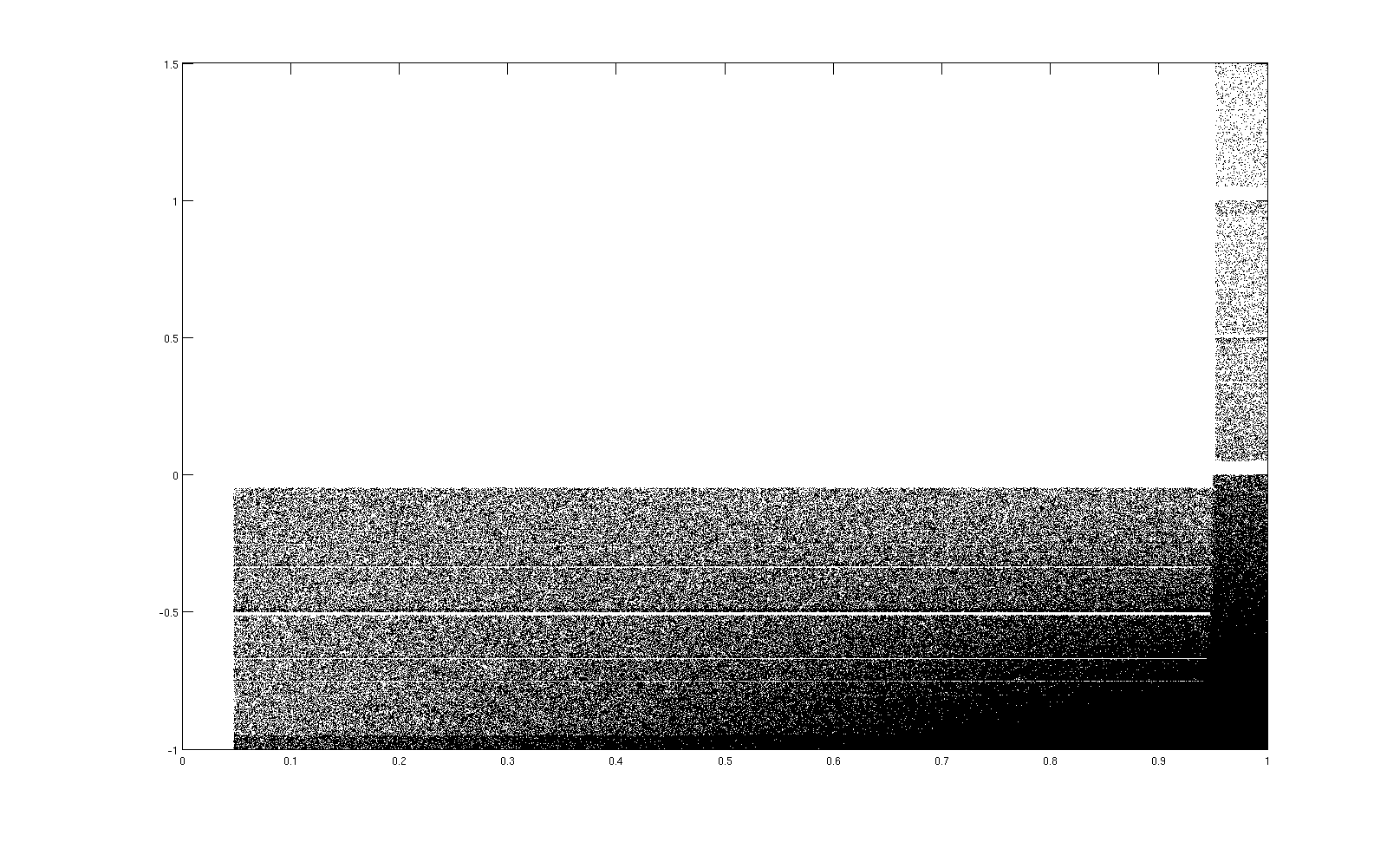}}
\end{subfigure}
\caption{Numerical simulations of $\mathcal{D}_\alpha$ for $\alpha > \frac{1}{2}\sqrt{2}$.}
\label{f:nesimulation}
\end{figure}

\section{Entropy, wandering rate and isomorphisms}\label{s:iso}
With an explicit expression for the density of $\mu_\alpha$ at hand, we can compute several dynamical quantities associated to the systems $T_\alpha$. In this section we compute the Krengel entropy, return sequence and wandering rate of $T_\alpha$ for a large part of the parameter space $(0,1)$.

\vskip .2cm
In \cite{Kre67} Krengel extended the notion of metric entropy to infinite, measure preserving and conservative systems $(X, \mathcal{B}, \mu, T)$ by considering the metric entropy on finite measure induced systems. More precisely, if $A$ is a sweep-out set for $T$ with $\mu(A) < \infty$, $T_{ A}$ the induced transformation of $T$ on $A$ and $\mu_A$ the restriction of $\mu$ to $A$, then the {\em Krengel entropy} of $T$ is defined to be
\begin{equation}\nonumber
h_{\text{Kr}, \mu}(T)=h_{\mu_A}(T_{A}, \mu_A),
\end{equation}
where $h_{\mu_A}(T_{A}, \mu_A)$ is the metric entropy of the system $(A, \mathcal B \cap A, T_A, \mu_A)$. Krengel proved in \cite{Kre67} that this quantity is independent of the choice of $A$. In \cite[Theorem 6]{Zwe00} it is shown that if $T$ is an AFN-map the Krengel entropy can be computed using Rohlin's formula:
\begin{equation}\label{e:ro}
h_{\text{Kr}, \mu}(T)= \int_X \log (|T'|) d \mu.
\end{equation}

The following theorem follows from Lemma~\ref{l:afn}, \eqref{q:density} and Table~\ref{t:densities}.
\begin{thm}\label{t:entropy}
For any $\alpha \in \big(0, g \big]$ the system $(I_{\alpha}, \mathcal{B}_{\alpha}, \mu_{\alpha}, T_{\alpha})$ has $h_{\text{Kr}, \mu_\alpha}(T_\alpha) = \frac{\pi^2}{6}$.
\end{thm}

\begin{proof}First fix $\alpha \in \big(0,\frac12\big)$. By Lemma~\ref{l:afn} we can use formula \eqref{e:ro} to compute the Krengel entropy of $T_\alpha$. For this computation we use some properties of the dilogarithm function, which is defined by
$$\Li_2(x):= \sum_{n=1}^{\infty} \frac{x^n}{n^2} \quad \text{ for } |x| \leq 1,$$
and satisfies (see \cite{Lew} for more information)
\begin{itemize}
\item $\Li_2(0)=0$;
\item $\Li_2(-1)= -\pi^2/12$;
\item $\Li_2(x)+\Li_2(-\frac{x}{1-x})= -\frac{1}{2}\log^2(1-x)$.
\end{itemize}
Using the density from \eqref{q:density} and these three properties of $\Li_2$ we get
\begin{align*}
\int_{I_\alpha} \log (|T'_{\alpha}|) d \mu_{\alpha} &= -2 \bigg(\int_{\alpha}^ {\frac{\alpha}{1-\alpha}} \frac{\log x}{x} dx + \int_{\frac{\alpha}{1-\alpha}}^{1-\alpha} \frac{\log x}{1+x} dx  +2 \int_{1-\alpha}^{1} \frac{\log x}{1-x^2} dx \bigg) \\
&= [-\log^2 x]_{\alpha}^{\frac{\alpha}{1-\alpha}} -2  [\Li_2(-x)+ \log x \log (x+1)]_{\frac{\alpha}{1-\alpha}}^{1-\alpha} \\
& \quad -2  [\Li_2(1-x)+\Li_2(-x)+ \log x \log (x+1)]_{1-\alpha}^{1} \\
&= -\log^2 \Big(\frac{\alpha}{1-\alpha}\Big)+\log^2 (\alpha) +2 \Li_2 \Big(\frac{-\alpha}{1-\alpha} \Big) + 2 \log \Big(\frac{\alpha}{1-\alpha}\Big) \log \Big(\frac{1}{1-\alpha}\Big)  \\
& \quad -2 \Li_2 (-1) + 2 \Li_2(\alpha)\\
&=\log^2 (\alpha) -\log^2 \big(\frac{\alpha}{1-\alpha}\big) -\log^2 (1-\alpha) - 2 \log \big(\frac{\alpha}{1-\alpha}\big) \log (1-\alpha)+\frac{\pi^2}{6} \\
&= \log^2 (\alpha) -\log^2 (\alpha) + 2\log (\alpha) \log (1-\alpha)  -2 \log^2 (1-\alpha) + \\
& \quad - 2 \log (\alpha)\log(1-\alpha) + 2\log^2(1-\alpha)+\frac{\pi^2}{6} \\
&= \frac{\pi^2}{6}.
\end{align*}
A similar computation yields $h_{\text{Kr}, \mu_\alpha}(T_\alpha) = \frac{\pi^2}{6}$ for $\alpha \in [\frac12, g]$.
\end{proof}

\begin{nrem}\label{r:krengel}{\rm
Numerical evidence using the densities from Table~\ref{t:densities} suggests that $h_{\text{Kr}, \mu_\alpha}(T_\alpha) = \frac{\pi^2}{6}$ for $\alpha \in (g, \frac12 \sqrt{2})$ as well. Even though we were not able to calculate the Krengel entropy for $\alpha \in (g, \frac12 \sqrt{2})$ explicitly, we conjecture that in fact $h_{\text{Kr}, \mu_\alpha}(T_\alpha) = \frac{\pi^2}{6}$ for all $\alpha \in (0, 1)$. This claim is supported by the fact that the Krengel entropy for Nakada's $\alpha$-continued fraction maps $S_\alpha$ from \eqref{q:nakada} is $\frac{\pi^2}{6}$ as well, see~\cite[Theorem 2]{KSS12}.}
\end{nrem}

\vskip .2cm
The {\em return sequence} of $T_{\alpha}$ is the sequence $(a_n(T_{\alpha}))_{n \ge 1}$ of positive real numbers satisfying~\eqref{e:pde}. The pointwise dual ergodicity of each map $T_{\alpha}$ implies that such a sequence, which is unique up to asymptotic equivalence, exists. The {\em asymptotic type} of $T_\alpha$ corresponds to the family of all sequences asymptotically equivalent to some positive multiple of $(a_n(T_{\alpha}))_{n \ge 1}$. The return sequence of a system is related to its wandering rate, which quantifies how big the system is in relation to its subsets of finite measure. To be more precise, if $(X, \mathcal{B}, \mu,T)$ is a conservative, ergodic, measure preserving system and $A \in \mathcal{B}$ a set of finite positive measure, then the wandering rate of $A$ with respect to $T$ is the sequence $(w_n(A))_{n \geq 1}$ given by
\[ w_n(A):= \mu \bigg( \bigcup_{k=0}^{n-1} T^{-k}A \bigg).\]
It follows from \cite[Theorem 2]{Zwe00} that for each of the maps $T_\alpha$ there exists a positive sequence $(w_n(T_\alpha))$ such that $w_n(T_\alpha) \uparrow \infty$ and $w_n(T_\alpha) \sim w_n(A)$ as $n \to \infty$ for all sets $A$ that have positive, finite measure and are bounded away from 1. The asymptotic equivalence class of $(w_n(T_\alpha))$ defines the {\em wandering rate} of $T_\alpha$. Using the machinery from \cite{Zwe00} and the explicit formula of the density we compute both the return sequence and the wandering rate of the maps $T_\alpha$.

\begin{prop}\label{prop:wander}
For all $\alpha\in(0,1)$ there is a constant $c_{\alpha} >0$ such that
\[w_n(T_\alpha) \sim c_{\alpha} \log n \quad \textit{and} \quad a_n(T_\alpha) \sim \frac{n}{c_{\alpha} \log n}. \]
If $\alpha\in(0,\frac{1}{2} \sqrt{2})$, then $c_{\alpha}=1$.
\end{prop}

\begin{proof}
Using the Taylor expansion of the maps $T_\alpha$, one sees that for $x \rightarrow 1$ we have $T_\alpha(x)= x-(x-1)^2+ o ((x-1)^2)$. Hence, $T_\alpha$ admits what are called {\em nice expansions} in \cite{Zwe00}. For $x \in \big(\frac1{1+\alpha}, 1 \big]$ we can write $f_\alpha(x) = \frac{x-2}{x-1} H(x)$, where the function $x \mapsto \frac{x-2}{x-1}$ corresponds to the map called $G$ in \cite[Theorem A]{Zwe00}. It then follows by \cite[Theorems 3 and 4]{Zwe00} that the wandering rate is
\begin{equation}\label{e:wrate}
 w_n(T_\alpha) \sim c_{\alpha} \log n
\end{equation}
and the return sequence is
\begin{equation}\label{e:rseq}
a_n(T_\alpha) \sim \frac{n}{c_{\alpha }\log n},
\end{equation}
for $c_{\alpha}= \lim_{x \uparrow 1} H(x)$. For $\alpha \in \big(0, \frac12 \sqrt{2}\big)$ the explicit formula for the densities from \eqref{q:density} and Table~\ref{t:densities} gives $c_{\alpha}=1$.
\end{proof}
We have now established all parts of Theorem~\ref{t:main1}.

\begin{proof}[Proof of Theorem~\ref{t:main1}]
The densities are given by~(\ref{q:density}) and listed in Table~\ref{t:densities}. The entropy is given by Theorem~\ref{t:entropy} and the wandering rate and return sequence in Proposition~\ref{prop:wander}.
\end{proof}

\begin{nrem}{\rm
(i) As in Remark~\ref{r:krengel} we suspect that in fact $c_\alpha=1$ for all $\alpha \in (0,1)$.

\vskip .1cm
\noindent (ii) Since all the results from \cite{Zwe00} apply to our family, we can use these to get an even more detailed description of the ergodic behaviour of the maps $T_\alpha$. We briefly mention a few more results for $\alpha \in (0, 1/2 \sqrt{2}]$. Since the return sequence $(a_n(T_\alpha))_{n \ge 1}$ is regularly varying with index $1$, by \cite[Theorem 5]{Zwe00} and \cite[Corollary 3.7.3]{Aar97}, we have
\begin{equation}\label{q:wll}
 \frac{\log n}{n} \sum_{k=0}^{n-1} f \circ T_\alpha^k \xrightarrow{\mu_{\alpha}} \int_{I_\alpha} f \, d\mu_\alpha, \quad \text{ for } f \in L^1(I_\alpha, \mathcal B_{\alpha}, \mu_\alpha) \text{ and } \int_{I_\alpha} f \, d\mu_\alpha \neq 0.
 \end{equation}
In other words,  a weak law of large numbers holds for $T_{\alpha}$.

\vskip .2cm
In addition, we can obtain asymptotics for the excursion times to the interval $\big[\frac{1}{1+\alpha},1\big]$, corresponding to the rightmost branch of $T_{\alpha}$. Let $Y$ be a sweep-out set, $T_Y$ the induced map on $Y$ and $\varphi: x \mapsto \min \{n \geq 1: T^n(x) \in Y\}$ the first return map. Write $\varphi_n^Y:=\sum_{k=0}^{n-1} \varphi \circ T_Y^k$ and note that the asymptotic inverse of the sequence $(a_n(T_\alpha))_{n \ge 1}$ is $(n\log n)_{n \ge 1}$,  so that the statement from \eqref{q:wll} is equivalent to the following dual:
$$\frac{1}{n\log n} \varphi_n^Y \xrightarrow{\mu_\alpha} \frac{1}{\mu_\alpha(Y)}.$$
If we induce on $Y:=\big [\min \{\alpha, 1-\alpha\}, \frac{1}{1+\alpha}\big]$, then $\varphi_n^Y$ sums the lengths (increased by $n$) of the first $n$ blocks of consecutive digits $(\epsilon,d)=(-1,2)$, and we obtain
$$\frac{1}{n \log n} \varphi_n^Y - \frac{1}{\log n}\xrightarrow{\mu_{\alpha}} \frac{1}{\mu_{\alpha}(Y)}.$$
From Theorem~\ref{t:main1}, it follows that for $\alpha < \frac12$, $\mu_{\alpha}(Y)=\log(2+\alpha)$. Note that for $\alpha$ decreasing the right hand side is increasing, meaning we spend on average more time in $\Delta(-1,2)$. Intuitively, for a smaller $\alpha$, every time we enter $\Delta(-1,2)$ we are closer to the indifferent fixed point, and it takes longer before we manage to escape from it.
}\end{nrem}

Note that the Krengel entropy, return sequence and wandering rate we found do not display any dependence on $\alpha$. These quantities give isomorphism invariants for dynamical systems with infinite invariant measures. Two measure preserving dynamical systems $(X,\mathcal B, \mu, T)$ and $(Y, \mathcal C, \nu, S)$ on $\sigma$-finite measure spaces are called {\em $c$-isomorphic} for $c \in (0, \infty]$ if there are sets $N \in \mathcal B$, $M \in \mathcal C$ with $\mu(N)=0=\nu(M)$ and $T(X \setminus N) \subseteq X \setminus N$ and $S(Y \setminus M) \subseteq Y \setminus M$ and if there is a map $\phi: X \setminus N \to Y \setminus M$ that is invertible, bi-measurable and satisfies $\phi \circ T = S \circ \phi$ and $\mu \circ \phi^{-1} = c \cdot \nu$. Invariants for $c$-isomorphisms are the asymptotic proportionality classes of the return sequence (see \cite[Propositions 3.7.1 and 3.3.2]{Aar97} and~\cite[Remark 8]{Zwe00}) and the normalised wandering rates, which combine the Krengel entropy with the wandering rates (see e.g.~\cite{Tha83, Zwe00}). It follows from Theorem~\ref{t:main1} that all these quantities are equal for all $T_\alpha$, $\alpha \in (0, \frac12 \sqrt 2)$. Using the idea from \cite{Kal}, however, we find many pairs $\alpha$ and $\alpha'$ such that $T_{\alpha}$ and $T_{\alpha'}$ are not $c$-isomorphic for any $c \in (0, \infty]$. Consider for example any $\alpha \in \big( \sqrt 2-1, \frac12 \big)$, so that $\alpha \in \big( \frac1{2+\alpha}, \frac12 \big)$, and any $\alpha' \in \big( \frac13, \frac{3-\sqrt 5}{2} \big)$, so that $T_{\alpha'}(\alpha') > 1-\alpha'$, see Figure~\ref{f:cisom}. For a contradiction, suppose that there is a $c$-isomorphism $\phi: I_\alpha \to I_{\alpha'} $ for some $c \in (0, \infty]$. Let $J=[\alpha,  1-\alpha]$ and note that any $x \in J$ has precisely one pre-image under $T_{\alpha}$. Since $\phi \circ T_\alpha = T_{\alpha'} \circ \phi$ and $\phi$ is invertible, any element of the set $\phi(J)$ must also have precisely one pre-image. Since $T_{\alpha'}(\alpha') > 1-\alpha'$, there are no such points, so $\mu_{\alpha'} (\phi(J))=0$. On the other hand, since $J$ is bounded away from 1, it follows that $0 < \mu_\alpha(J) < \infty$. Hence, there can be no $c$, such that $\mu_{\alpha'} \circ \phi^{-1} = c \cdot \mu_\alpha$. Obviously a similar argument holds for many other combinations of $\alpha$ and $\alpha'$, even for $\alpha > \frac12$, and in case the argument does not work for $T_\alpha$ and $T_{\alpha'}$, one can also consider iterates $T_\alpha^n$, $T_{\alpha'}^n$. Hence, even though the above discussed isomorphism invariants are equal for all $\alpha \in \big( 0, \frac12 \big)$, it is not generally the case that any two maps $T_\alpha$ are $c$-isomorphic. We conjecture that for almost all pairs $(\alpha, \alpha')$, the maps $T_\alpha$ and $T_{\alpha'}$ are not $c$-isomorphic. 

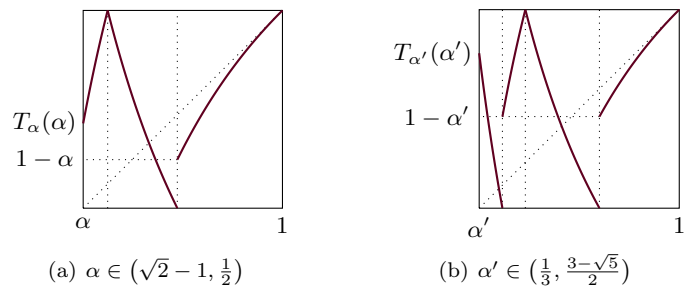
\begin{figure}[h]
\begin{center}
\subfigure[$\alpha \in \big( \sqrt 2-1, \frac12 \big)$]{
\begin{tikzpicture}[scale=4.6]
\draw(.43,.43)node[below]{\small $\alpha$}--(.6,.43)node[below]{\color{white}\small $\alpha'$}--(1,.43)node[below]{\small $1$}--(1,1)--(.43,1)--(.43,.43);
\draw[dotted] (.6993,.43)--(.6993,1)(.5,.43)--(.5,1)(.43,.43)--(1,1)(.6993,.57)--(.43,.57)node[left]{\small $1-\alpha$};
\draw[thick, purple!50!black, smooth, samples =20, domain=.6993:1] plot(\x,{2-1 / \x });
\draw[thick, purple!50!black, smooth, samples =20, domain=.5:.6993] plot(\x,{1 / \x -1});
\draw[thick, purple!50!black, smooth, samples =20, domain=.43:.5] plot(\x,{3-1 / \x });
\node at (.32,.6744) {\small $T_\alpha(\alpha)$};
\end{tikzpicture}}
\hspace{.8cm}
\subfigure[$\alpha' \in \big( \frac13, \frac{3-\sqrt 5}{2} \big)$]{
\begin{tikzpicture}[scale=4.05]
\draw(.35,.35)node[below]{\small $\alpha'$}--(1,.35)node[below]{\small $1$}--(1,1)--(.35,1)--(.35,.35);
\draw[dotted] (.7407,.35)--(.7407,1)(.5,.35)--(.5,1)(.35,.35)--(1,1)(.4255,.35)--(.4255,1)(.7407,.65)--(.35,.65)node[left]{\small $1-\alpha'$};
\draw[thick, purple!50!black, smooth, samples =20, domain=.7407:1] plot(\x,{2-1 / \x });
\draw[thick, purple!50!black, smooth, samples =20, domain=.5:.7407] plot(\x,{1 / \x -1});
\draw[thick, purple!50!black, smooth, samples =20, domain=.4255:.5] plot(\x,{3-1 / \x });
\draw[thick, purple!50!black, smooth, samples =20, domain=.35:.4255] plot(\x,{1 / \x -2});
\node at (.2,.857) {\small $T_{\alpha'}(\alpha')$};
\end{tikzpicture}}
\end{center}
\caption{Maps $T_\alpha$ and $T_{\alpha'}$ that are not $c$-isomorphic for any $c \in \mathbb (0, \infty]$.}
\label{f:cisom}
\end{figure}

\section{Acknowledgment}
The third author is supported by the NWO TOP-Grant No.~614.001.509.

\end{document}